\documentclass[sn-mathphys]{sn-jnl}% Math and Physical Sciences Reference Style
\usepackage{amsmath,amscd}
\usepackage{fancyhdr}

\numberwithin{equation}{section}

\jyear{2021}%

\theoremstyle{thmstyleone}%
\newtheorem{theorem}{Theorem}[section]%  meant for continuous numbers
\newtheorem{proposition}[theorem]{Proposition}%
\newtheorem{remark}[theorem]{Remark}

\begin{document}

\title[Monotonicity Conjectures and Sharp stability of Solitons ]{Monotonicity Conjectures and Sharp Stability for Solitons of the Cubic-Quintic NLS on $\mathbb{R}^3$}


\author*[1,3]{\fnm{Jian} \sur{Zhang}}\email{zhangjian@uestc.edu.cn}

\author[2]{\fnm{Chenglin} \sur{Wang}}\email{wangchenglinedu@163.com}

\author[3]{\fnm{Shihui} \sur{Zhu}}\email{shihuizhumath@163.com}

\affil*[1]{\orgdiv{School of Mathematical Sciences}, \orgname{University of Electronic Science and Technology of China}, \orgaddress{\street{} \city{Chengdu}, \postcode{611731}, \state{} \country{China}}}

\affil[2]{\orgdiv{School of Science}, \orgname{Xihua University}, \orgaddress{\street{} \city{Chengdu}, \postcode{610039}, \state{} \country{China}}}

\affil[3]{\orgdiv{School of Mathematical Sciences}, \orgname{Sichuan Normal University}, \orgaddress{\street{} \city{Chengdu}, \postcode{610066}, \state{} \country{China}}}

\abstract{This paper deals with the cubic-quintic nonlinear Schr\"{o}dinger equation on $\mathbb{R}^{3}$. Two monotonicity conjectures for solitons posed by Killip, Oh, Pocovnicu and Visan are completely resolved: one concerning frequency monotonicity, and the other concerning mass monotonicity. Uniqueness of the energy minimizer is proved. Then sharp stability of the solitons is established. And classification of normalized solutions is first presented. }


\pacs[MSC classfication(2020)]{35Q55, 35C08, 37K40, 35J50, 35J60}

\maketitle
\newpage
\tableofcontents % ²åÈëĿ¼

\pagestyle{fancy} % »Ö¸´Õý³£µÄҳüÉèÖÃ
\fancyhf{}
\fancyhead[OL]{ }
\fancyhead[OC]{}
\fancyhead[OR]{J. Zhang, C. Wang and S.  Zhu}
\fancyhead[EL]{Monotonicity Conjectures and Sharp stability of Solitons }
\fancyhead[EC]{ }
\fancyhead[ER]{ }
\fancyfoot[RO, LE]{\thepage}

\section{Introduction}\label{sec1}

This paper studies the cubic-quintic nonlinear Schr\"{o}dinger equation (NLS)
\begin{equation}\label{1.1}
i\partial_{t}\varphi +\Delta  \varphi +\lvert\varphi\rvert^2\varphi-\lvert\varphi\rvert^4\varphi=0,\quad (t,x)\in\mathbb{R}\times\mathbb{R}^3.
\end{equation}
(\ref{1.1}) is the defocusing energy critical  nonlinear Schr\"{o}dinger equation with focusing cubic perturbation on $\mathbb{R}^3$. We are interested in orbital stability of solitons at every frequency as well as the monotonicity conjectures posed by Killip, Oh, Pocovnicu and Visan \cite{KOPV2017}.

(\ref{1.1})  is a typical model for soliton theory (see \cite{KOPV2017,M2019}) in physics, such as nonlinear optics (see \cite{MMCTBMT2002,MMCML2000,DMM2000}), Langmuir waves in plasma (see \cite{ZH1994}) and the Bose-Einstein condensation in macroscopic quantum mechanics (see \cite{GFTLC2000,AGTF2001}).

From viewpoint of mathematics, the focusing cubic or quintic nonlinear Schr\"{o}dinger equation on $\mathbb{R}^3$ possesses  solitons, but all solitons are unstable (see \cite{BC1981,S2009,KM2006,NS2012,OT1991}). The defocusing cubic or quintic nonlinear Schr\"{o}dinger equation on $\mathbb{R}^3$ has no any solitons to exist, but possesses scattering property (see \cite{B1999,CKSTT2008,D2012,DHR2008}). In addition, the scaling invariance of the pure power case is breaked in (\ref{1.1}) (see \cite{C2003}). Therefore the dynamics of (\ref{1.1}) becomes a challenging issue (see \cite{KOPV2017,CS2021,MXZ2013,LR2020,TVZ2007,JL2022,AM2022,MKV2021}). This motivates us to study the stability of solitons for (\ref{1.1}) (see \cite{T2009}), which  directly  concerns about the monotonicity conjectures posed by Killip, Oh, Pocovnicu and Visan (see Conjecture 2.3 and 2.6 in \cite{KOPV2017}) as well as Lewin and Rota Nodari \cite{LR2020}.

In the energy space $H^{1}(\mathbb{R}^3)$, consider the scalar field equation for $\omega\in\mathbb{R}$,
\begin{equation}\label{1.2}
-\Delta u-\lvert u\rvert^2u+\lvert u\rvert^4u+\omega u =0, \quad u\in H^{1}(\mathbb{R}^3).
\end{equation}
From Berestycki and Lions \cite{BL1983}, if and only if
\begin{equation}\label{1.3}
0<\omega<\frac{3}{16},
\end{equation}
(\ref{1.2}) possesses non-trivial solutions (also see \cite{KOPV2017}). From Gidas, Ni and Nirenberg \cite{GNN1979}, every positive solution of (\ref{1.2}) is radially symmetric. From Serrin and Tang \cite{ST2000}, the positive solution of (\ref{1.2}) is unique up to translations. Therefore one concludes that for $\omega\in(0,\frac{3}{16})$, (\ref{1.2}) possesses a unique positive solution $P_\omega(x)$ (see \cite{KOPV2017,K2011}), which is called ground state of (\ref{1.2}).

Let $P_\omega$ be a ground state of (\ref{1.2}) with $\omega\in(0,\frac{3}{16})$. It is easily checked that
\begin{equation}\label{1.4}
\varphi(t,x)=P_\omega(x)e^{i\omega t}
\end{equation}
is a solution of (\ref{1.1}), which is called a ground state soliton of (\ref{1.1}). We also directly call (\ref{1.4}) soliton of (\ref{1.1}), and call $\omega$ frequency of soliton. It is known that (\ref{1.1}) admits time-space translation invariance, phase invariance and Galilean invariance. Then for arbitrary $x_0\in\mathbb{R}^3$, $v_0\in\mathbb{R}^3$ and $\nu_0\in\mathbb{R}$, in terms of (\ref{1.4}) one has that
\begin{equation}\label{1.5}
 \varphi(t,x)=P_\omega(x-x_0-v_0t)e^{i(\omega t+\nu_0+\frac{1}{2}v_0x-\frac{1}{4}\lvert v_0\rvert^2t)}
\end{equation}
is also a soliton of (\ref{1.1}).
By \cite{KOPV2017,S2021,CS2021,JJTV2022,JL2022,LR2020}, orbital stability of solitons with regard to every frequency  for (\ref{1.1}) is a crucial open problem.

So far there are two ways to study stability of solitons for nonlinear  Schr\"{o}dinger equations (refer to \cite{L2009}). One is variational approach originated from Cazenave and Lions \cite{CL1982}. The other is spectrum approach originated from Weinstein \cite{W1985,W1986} and then considerably generalized by Grillakis, Shatah and Strauss \cite{GSS1987,GSS1990}. Both approaches have encountered essential difficulties to (\ref{1.1}), since (\ref{1.1}) fails in both scaling invariance and effective spectral analysis \cite{F2003,O1995,CS2021}. We need develop new methods to study stability of solitons for (\ref{1.1}).

In the following, we denote $\int_{\mathbb{R}^3}\cdot\ dx$ by $\int \cdot \ dx$. For $u\in H^1(\mathbb{R}^3)$, define the mass functional
\begin{equation}\label{1.11}
M(u)=\int\lvert u\lvert^{2}dx,
\end{equation}
and define the energy functional
\begin{equation}\label{1.12}
E(u)=\int \frac{1}{2}\lvert \nabla u\rvert^2-\frac{1}{4}\lvert u\rvert^4+\frac{1}{6}\lvert u\rvert^6dx.
\end{equation}

Our studies are also concerned with the normalized solution of (\ref{1.2}), which is defined as a non-trivial solution of (\ref{1.2}) satisfying the prescribed mass $M(u)=m$ (see \cite{JL2021}). Besides motivations in mathematical physics, normalized solutions are also of interest in the framework of ergodic Mean Field Games system \cite{CV2017}. Recent studies on normalized solutions refer to \cite{BJS2016,BMRV2021,JL2021,WW2022,S2021} and the references there. We can establish correspondences between the soliton frequency and the prescribed mass. Then we present first a classification of normalized solutions.

 For $0<\alpha<\infty$ and $u\in H^1(\mathbb{R}^3)\backslash\{0\}$, define the functional (see \cite{KOPV2017})
\begin{equation}\label{1.8}
F_{\alpha}(u)=\frac{\lvert\lvert u\rvert \rvert_{L^2}\lvert\lvert u\rvert \rvert_{L^6}^{\frac{3\alpha}{1+\alpha}}\lvert\lvert \nabla u\rvert \rvert_{L^2}^{\frac{3}{1+\alpha}}}
{\lvert\lvert u\rvert \rvert_{L^4}^{4}}
\end{equation}
and the variational problem
\begin{equation}\label{1.9}
C_{\alpha}^{-1}=\mathop{\mathrm{inf}}_{\{u\in H^1(\mathbb{R}^3)\backslash\{0\}\}} F_{\alpha}(u).
\end{equation}
In \cite{KOPV2017}, Killip, Oh, Pocovnicu and Visan show that (\ref{1.9}) is achieved at some positive symmetric minimizer $Q_{\alpha}$. Moreover $Q_{\alpha}$ satisfies the  Euler-Lagrange equation corresponding to (\ref{1.9}), which is the same as (\ref{1.2}) with certain $\omega\in(0,\frac{3}{16})$. Especially for $\alpha=1$, one has that
\begin{equation}\label{1.10}
m_{1}=\lvert\lvert Q_{1}\rvert \rvert_{L^2}^{2}=\frac{64}{9}C_{1}^{-2}.
\end{equation}
We see that uniqueness of $Q_{1}$ is unknown, but all $Q_{1}$ have the common mass $m_{1}$ (see \cite{KOPV2017}).

Now for $u\in H^{1}(\mathbb{R}^{3})$, we introduce the notation % $\omega\in (0,\frac{3}{16})$,
\begin{equation}\label{1.6}
\beta (u)=\frac{\int \rvert u\rvert^{6}dx}{\int \rvert\nabla u\rvert^{2} dx}.
\end{equation}
Especially, we denote $\beta(\omega)=\beta(P_{\omega})$ since the soliton $P_{\omega}$ is determined uniquely by the frequency $\omega$ (see \cite{KOPV2017,K2011}). Then fix $\omega\in (0,\frac{3}{16})$, define
\begin{equation}\label{1.7}
R_{\omega}(x)=\sqrt{\frac{1+\beta(\omega)}{4\beta(\omega)}}P_{\omega}\Big(\frac{3[1+\beta(\omega)]}{4\sqrt{3\beta(\omega)}}x\Big).
\end{equation}
We call $R_{\omega}$ rescaled soliton (see \cite{KOPV2017}).

By Killip, Oh, Pocovnicu and Visan \cite{KOPV2017}, for $0<\omega<\frac{3}{16}$ the ground state $P_{\omega}$ and the rescaled soliton $R_{\omega}$ satisfy
\begin{equation}\label{1.7a}
\int \lvert P_{\omega}(x)\lvert^{2}dx\geq \int \lvert R_{\omega}(x)\lvert^{2}dx\geq \frac{4}{3\sqrt{3}}\lvert\lvert Q_{1}\rvert \rvert_{L^2}^{2}.
\end{equation}

Let $m>0$, and consider the constrained variational problem
\begin{equation}\label{1.13}
E_{min}(m)=\mathop{\mathrm{inf}}_{\{u\in H^1(\mathbb{R}^3),\;\;  M(u)=m\}} E(u).
\end{equation}
For the common mass $m_{1}=M(Q_{1})$ in (\ref{1.10}), Killip, Oh, Pocovnicu and Visan \cite{KOPV2017} show the following results. When
\begin{equation}\label{1.14}
0<m<M(Q_{1}),
\end{equation}
(\ref{1.13}) is not achieved. When
\begin{equation}\label{1.15}
m\geq M(Q_{1}),
\end{equation}
(\ref{1.13}) is achieved.

For $u\in H^{1}(\mathbb{R}^{3})$, define the Pohozaev functional (refer to \cite{P1965})
\begin{equation}\label{1.16}
V(u)=\int \lvert \nabla u\lvert^{2}+\lvert u\lvert^{6}-\frac{3}{4}\lvert u\lvert^{4}
dx.
\end{equation}
Let $m>0$, define the cross-constrained variational problem
\begin{equation}\label{1.17}
E_{min}^{V}(m)=\inf_{\{u\in H^{1}(\mathbb{R}^{3}), \; M(u)=m, \; V(u)=0\}}E(u).
\end{equation}
For the common mass $M(Q_{1})$ in (\ref{1.10}), Killip, Oh, Pocovnicu and Visan \cite{KOPV2017} get the following results. When
\begin{equation}\label{1.18}
0<m<\frac{4}{3\sqrt{3}}M(Q_{1}) ,
\end{equation}
one has that $E_{min}^{V}(m)=\infty$ and (\ref{1.17}) is not achieved. When
\begin{equation}\label{1.19}
m\geq \frac{4}{3\sqrt{3}} M(Q_{1}),
\end{equation}
one has that $-\infty<E_{min}^{V}(m)<\infty$ and (\ref{1.17}) is achieved at some positive minimizer $\psi$. We call this $\psi$ energy minimizer for (\ref{1.1}).

The main results of this paper read as follows:

\begin{theorem}\label{t1.1}
Let $0<\omega <\frac{3}{16}$. Then $0<\beta(\omega)<\infty$. Moreover $\beta(\omega)$ is strictly increasing on $\omega\in (0,\frac{3}{16})$.
\end{theorem}

\begin{theorem}\label{t1.2}
Let $M(P_{\omega})=\int P_{\omega}^{2}dx$ for the positive solution $P_{\omega}$ of (\ref{1.2}) with $\omega\in (0, \frac{3} {16})$. Then there is an $\omega_*\in (0, \frac{3} {16})$ so that the map $\omega\rightarrow M(P_{\omega})$ is strictly decreasing for $\omega<\omega_{\ast}$ and strictly increasing for $\omega>\omega_{\ast}$.
\end{theorem}

\begin{theorem}\label{t1.3}
Let $m\geq\frac{4}{3\sqrt{3}}M(Q_{1})$. Then the cross-constrained variational problem (\ref{1.17}) possesses a unique positive minimizer up to translations.
\end{theorem}

\begin{theorem}\label{t1.4}
Let $\omega\in (0, \frac{3} {16})$ and $P_{\omega}(x)$ be the positive solution of (\ref{1.2}). Then there exists a unique $\omega_{\ast}$ such that $0<\omega_{\ast}<\frac{3}{16}$ and $\beta(\omega_{\ast})=\frac{1}{3}$. When $\omega\in (0, \omega_{\ast}]$, the soliton $P_{\omega}(x)e^{i\omega t}$ of (\ref{1.1}) is orbitally unstable; When $\omega\in (\omega_{\ast}, \frac{3}{16})$, the soliton $P_{\omega}(x)e^{i\omega t}$ of (\ref{1.1}) is orbitally stable.
\end{theorem}

\begin{theorem}\label{t1.5}
Let $m_{0}=M(P_{\omega_{\ast}})$. Then, when $m=m_{0}$, (\ref{1.2}) has a unique positive normalized solution with the prescribed mass $M(u)=m_{0}$, which is just $P_{\omega_{\ast}}$. When $m> m_{0}$, (\ref{1.2}) just has two positive normalized solutions with the prescribed mass $M(u)=m$. The one is $P_{\omega_{1}}$ with $\omega_{1}\in (0, \omega_{\ast})$ and the other is $P_{\omega_{2}}$ with $\omega_{2}\in (\omega_{\ast}, \frac{3}{16})$. When $m<m_{0}$, (\ref{1.2}) has no positive normalized solution with the prescribed mass $M(u)=m$.
\end{theorem}

Theorem \ref{t1.1} completely resolves Conjecture 2.6 posed by
Killip, Oh, Pocovnicu and Visan (see \cite{KOPV2017}), which is concerning the monotonicity of frequency.

Theorem \ref{t1.2} completely resolves Conjecture 2.3 posed by
Killip, Oh, Pocovnicu and Visan (see \cite{KOPV2017}), which is concerning the monotonicity of mass. In \cite{LR2020}, Lewin and Rota Nodari get the exact behavior of the mass $M(P_{\omega})$ and its derivative $\frac{d}{d\omega}M(P_{\omega})$, and show the monotonicity close to both $0$ and $\frac{3}{16}$. They give a proof of the monotonicity conjecture at large mass. In fact, Lewin and Rota Nodari \cite{LR2020} consider more general situations. Therefore Lewin and Rota Nodari improve a lot about the monotonicity conjecture and make some advances. The related studies we also refer to Carles and Sparber \cite{CS2021} as well as Jeanjean et al \cite{JJTV2022,JL2021,JL2022}.

Theorem \ref{t1.3} gets uniqueness of the positive minimizer for $E_{min}^{V}(m)$ as well as $E_{min}(m)$. Thus Theorem \ref{t1.3} settles the questions raised by \cite{JL2022} and \cite{KOPV2017}, where uniqueness of the energy minimizer is proposed. We see that uniqueness of the energy minimizer is concerned with the monotonicity of $M(P_{\omega})$ for $\omega \in (0,\frac{3}{16})$ and stability of solitons. In fact, uniqueness of positive minimizer of the variational problem is remarkable \cite{LR2020,GZ2008,Z2000,Z2002,Z2005}.

Theorem \ref{t1.4} establishes sharp results for orbital stability of solitons of (\ref{1.1}) with regard to every frequency by the critical frequency $\omega_{\ast}$. Therefore Theorem \ref{t1.4} settles the questions raised by \cite{KOPV2017,S2021,CS2021,JJTV2022,JL2022,LR2020}, where the orbitally stability of solitons is only with regard to the set of ground states. Then Theorem \ref{t1.4} also answers the open problem on stability of solitons proposed by Tao \cite{T2009}.

Theorem \ref{t1.5} first gives a complete classification about normalized solutions of (\ref{1.2}), which depends on establishing two correspondences between the frequency and the mass. Thus Theorem \ref{t1.5} settles the questions raised by \cite{BMRV2021,JL2021, WW2022, S2021}.

We see that (\ref{1.1}) is almost a perfect nonlinear dispersive model, which is global well-posedness in energy space and possesses scattering, stable solitons as well as unstable solitons. By \cite{KOPV2017} and Theorem \ref{t1.4} in this paper, the global dynamics of (\ref{1.1}) is more comprehensively described (also see \cite{AM2022,MKV2021}). According to \cite{CMM2011,MM2006,MMT2006}, multi-solitons of (\ref{1.1})
can be constructed (also refer to \cite{CL2011,BZ2022}). But the soliton resolution conjecture for (\ref{1.1}) is still open (see \cite{T2009, M2024}).

This paper is organized as follows. In section 2, we induce some propositions about solitons and rescaled solitons based on \cite{KOPV2017}. In section 3, we state some known results and new results according to variational approaches. Especially, Theorem \ref{t3.9} claims an essential new uniqueness result. In section 4, we prove some key monotonicity and uniqueness results in case $\alpha\geq 1$. In section 5, we prove some key uniqueness and monotonicity results in case $\frac{1}{3}\leq \alpha<1$. In section 6, we complete the proof of monotonicity of the frequency. In section 7, we complete the proof of monotonicity of the mass. In section 8, we prove uniqueness of energy minimizer. In section 9, we establish sharp stability threshold of solitons for (\ref{1.1}). In section 10, we give classification of normalized ground states of (\ref{1.2}) for the first time.

\section{Ground states and rescaled solitons}\label{sec2}

According to  Killip, Oh, Pocovnicu and Visan \cite{KOPV2017}, the following propositions are true.

\begin{proposition}\label{p2.1}
(\cite{KOPV2017}) (\ref{1.2}) possesses a positive solution $P_\omega$ if and only if $\omega\in(0,\frac{3}{16})$. In addition, $P_\omega$  holds the following properties.\\
(I) $P_\omega$ is radially symmetric and unique up to translations.\\
(II) $P_\omega(x)$ is a real-analytic function of $x$ and that for some $c=c(\omega)>0$ as $\lvert x\rvert\rightarrow \infty$,
\begin{equation*}
\lvert x\rvert \mathrm{exp}\{\sqrt{\omega}\lvert x\rvert\}P_\omega(x)\rightarrow c,\quad
\mathrm{exp}\{\sqrt{\omega}\lvert x\rvert\}x \nabla P_\omega(x)\rightarrow -\sqrt{\omega}c.
\end{equation*}
(III) Nehari identity and Pohozaev identity:
\begin{equation*}
\int \lvert \nabla P_\omega \rvert^2+\omega \rvert P_\omega \rvert^2+\rvert P_\omega \rvert^6-\rvert P_\omega \rvert^4dx=0,
\end{equation*}
\begin{equation*}
\int \frac{1}{3}\lvert \nabla P_\omega \rvert^2 + \frac{1}{3}P_\omega^6-\frac{1}{4}P_\omega^4dx=0.
\end{equation*}
(IV) \begin{equation*}
\frac{d}{d\omega}E(P_\omega)=-\frac{\omega}{2}\frac{d}{d\omega}M(P_\omega),\quad
\frac{d}{d\omega}\int \lvert \nabla P_\omega \rvert^2dx=\frac{3}{2}M(P_\omega).
\end{equation*}
\end{proposition}
\begin{proposition}\label{p2.2}
(\cite{KOPV2017}) Let $\omega\in(0,\frac{3}{16})$ and $P_\omega$ be the ground state of (\ref{1.2}). Put $\beta(\omega)=\int P_\omega ^6dx$ / $\int \lvert \nabla P_\omega \rvert^2dx$, then \\
~~~~~~~~~~\\
(I)
\begin{equation*}
E(P_\omega)=\frac{1-\beta(\omega)}{6}\int\lvert \nabla P_\omega\rvert ^2dx,
\end{equation*}
\begin{equation*}
M(P_\omega)=\frac{\beta(\omega)+1}{3\omega}\int \lvert \nabla P_\omega \rvert^2dx,
\end{equation*}
\begin{equation*}
\frac{d}{d\omega}M(P_{\omega})<\frac{3\beta(\omega)-1}{2\omega}M(P_{\omega}).
\end{equation*}
(II) As $\omega\rightarrow 0$ we have $\beta(\omega)=\omega \beta(g)+O(\omega^{2})$,
\begin{equation*}
M(P_{\omega})=\frac{1}{\sqrt{\omega}}\int g(x)^{2}dx+\frac{\sqrt{\omega}}{2}\int g(x)^{6}dx+O(\omega^{3/2}),
\end{equation*}
and
\begin{equation*}
E(P_{\omega})=\frac{\sqrt{\omega}}{2}\int g(x)^{2}dx-\frac{\omega^{3/2}}{12}\int g(x)^{6}dx+O(\omega^{5/2}),
\end{equation*}
where $g$ is the unique positive radially symmetric solution to $-\Delta g-g^{3}+g=0$.\\
(III) As $\omega\rightarrow \frac{3}{16}$ we have $\beta(\omega)\rightarrow\infty$, $M(P_{\omega})\rightarrow\infty$, and $E(P_{\omega})\rightarrow -\infty$; Indeed,
\begin{equation*}
\beta(\omega)\sim (\frac{3}{16}-\omega)^{-1},\quad M(P_{\omega})\sim (\frac{3}{16}-\omega)^{-3}, \quad \rvert E(P_{\omega})\rvert\sim (\frac{3}{16}-\omega)^{-3}.
\end{equation*}
\end{proposition}

\begin{proposition}\label{p2.3}
(\cite{KOPV2017}) Fix $0<\omega<\frac{3}{16}$. Among all rescaling $u(x)=aP_{\omega}(\lambda x)$ of $P_{\omega}$ with $a>0$ and $\lambda>0$ there is exactly one that obeys $\beta(u)=\frac{1}{3}$ and $V(u)=0$, namely,
\begin{equation*}
R_{\omega}(x)=\sqrt{\frac{1+\beta(\omega)}{4\beta(\omega)}}P_{\omega}\Big(\frac{3[1+\beta(\omega)]}{4\sqrt{3\beta(\omega)}}x\Big).
\end{equation*}
Moreover,
\begin{equation*}
E(R_{\omega})=\frac{1}{9\sqrt{3\beta(\omega)}}\int \lvert\nabla P_{\omega}\lvert^{2}dx,\quad M(R_{\omega})=\frac{16\sqrt{3\beta(\omega)}}{9[1+\beta(\omega)]^{2}}M(P_{\omega}),
\end{equation*}
$M(R_{\omega})\leq M(P_{\omega})$ with equality if and only if $\beta(\omega)=\frac{1}{3}$,\\
$E(R_{\omega})\geq  E(P_{\omega})$ with equality if and only if $\beta(\omega)=\frac{1}{3}$,\\
$[\beta(\omega)-1]\frac{d}{d\omega}M(R_{\omega})\geq 0 $ with equality if and only if $\beta(\omega)=1$,\\
\begin{equation*}
\frac{d}{d\omega}E(R_{\omega})>0,
\end{equation*}
\begin{equation*}
\lim_{\omega\searrow 0}E(R_{\omega})=\frac{1}{3\sqrt{3\beta(g)}}M(g),\quad \text{and}\quad \lim_{\omega\searrow 0}M(R_{\omega})=\frac{16\sqrt{3\beta(g)}}{9}M(g),
\end{equation*}
where $g$ is the unique positive radial solution to $-\Delta g-g^{3}+g=0$.
\end{proposition}

\begin{remark}
For the one-dimensional cubic-quintic nonlinear Schr\"{o}dinger equation, the corresponding ground state satisfies the following one-dimensional nonlinear elliptic equation
\begin{equation*}
-\Delta \phi+\omega\phi -\phi^{3}+\phi^{5}=0,\;\;\;\;\phi\in H^{1}(\mathbb{R})\backslash\{0\},
\end{equation*}
where $\omega>0$ is frequency of the soliton. From \cite{CERS1986,PPT1979}, one has that
\begin{equation*}
\phi(x)=2\sqrt{\frac{\omega}{1+\sqrt{1-\frac{16}{3}\omega}\cdot cosh(2\sqrt{\omega}x)}}.
\end{equation*}
It is obvious that $\omega\in (0,\frac{3}{16})$.
\end{remark}

\section{Variational approaches }\label{sec3}

According to  Killip, Oh, Pocovnicu and Visan \cite{KOPV2017}, the following propositions are true.

\begin{proposition}\label{p3.1}
(\cite{KOPV2017}) Let $0<\alpha<\infty$. Then there exists a positive radially symmetric function $v(x)\in H^{1}(\mathbb{R}^{3})$ such that
\begin{equation*}
C_{\alpha}^{-1}=\mathop{\mathrm{min}}_{\{u\in H^1(\mathbb{R}^3)\backslash\{0\}\}} F_{\alpha}(u).
\end{equation*}
Define $Q_{\alpha}(x)=\lambda^{-1}v(x/\rho)$  where $\lambda,\; \rho>0$ are given by
\begin{equation*}
\lambda^{2}=\frac{4(1+\alpha)}{3\alpha}\frac{\int v^{6}dx}{\int v^{4}dx},\quad \rho^{2}=\frac{16(1+\alpha)^{2}}{9\alpha}\frac{\int v^{6}dx\int \lvert\nabla v\lvert^{2}dx}{(\int v^{4}dx)^{2}},
\end{equation*}
then, $Q_{\alpha}$ satisfies
\begin{equation*}
-\Delta Q_{\alpha}+Q_{\alpha}^{5}-Q_{\alpha}^{3}+\omega Q_{\alpha}=0\quad \text{with}\quad \omega=\frac{3\alpha}{16(1+\alpha)}\frac{(\int v^{4}dx)^{2}}{\int v^{2}dx \int v^{6}dx}. \end{equation*}
Moreover,
\begin{equation*}
\beta(Q_{\alpha})=\alpha, \quad C_{\alpha}=\frac{4(1+\alpha)}{3\alpha^{\frac{\alpha}{2(1+\alpha)}}}\frac{1}{\lvert\lvert Q_{\alpha}\lvert\lvert_{L^{2}}\lvert\lvert \nabla Q_{\alpha}\lvert\lvert_{L^{2}}^{\frac{1-\alpha}{1+\alpha}}}.
\end{equation*}
\end{proposition}

\begin{proposition}\label{p3.2}
(\cite{KOPV2017}) Let $\alpha=1$. Then there exists $Q_{1}$ such that $C_{1}=\frac{8}{3}\frac{1}{\lvert\lvert Q_{1}\lvert\lvert_{L^{2}}}$ and $\beta(Q_{1})=1$. Moreover, for $R_{1}(x)=\frac{1}{\sqrt{2}}Q_{1}(\frac{\sqrt{3}}{2}x)$, one has that
$M(R_{1})=\frac{4}{3\sqrt{3}}M(Q_{1})$.
\end{proposition}

\begin{proposition}\label{p3.3}
(\cite{KOPV2017}) Let $0<\alpha<\infty$. Then
\begin{equation*}
\alpha \mapsto \lvert\lvert \nabla Q_{\alpha}\lvert\lvert_{L^{2}}^{2}\quad \text{is strictly increasing on}\; (0,\infty).
\end{equation*}
\begin{equation*}
\alpha \mapsto M(Q_{\alpha})=\lvert\lvert Q_{\alpha}\lvert\lvert_{L^{2}}^{2} \quad \text{is strictly decreasing on}\;  (0,\frac{1}{3}].
\end{equation*}
\begin{equation*}
\alpha \mapsto M(Q_{\alpha})=\lvert\lvert Q_{\alpha}\lvert\lvert_{L^{2}}^{2} \quad \text{is strictly increasing on}\;  [1,\infty).
\end{equation*}
\end{proposition}

\begin{proposition}\label{p3.4}
(\cite{KOPV2017}) Let $0<\omega<\frac{3}{16}$. Then one has that
\begin{equation*}
M(P_{\omega})\geq M(R_{\omega})\geq \frac{4}{3\sqrt{3}}M(Q_{1}),
\end{equation*}
\begin{equation*}
M(P_{\omega})=M(R_{\omega})\;\;\;\;\text{if and only if}\;\;\;\beta(\omega)=\frac{1}{3}.
\end{equation*}
\end{proposition}

\begin{proposition}\label{p3.5}
(\cite{KOPV2017}) For the variational problem (\ref{1.13}), one has that:\\
(I) If $0<m<M(Q_{1})$, then
\begin{equation*}
E(u)>0 \quad \text{and} \quad E_{min}(m)=0.
\end{equation*}
In this case, $E_{min}(m)$ is not achieved.\\
(II) If $m=M(Q_{1})$, then
\begin{equation*}
E(u)\geq0 \quad \text{and} \quad E_{min}(m)=0.
\end{equation*}
In this case, $E_{min}(m)$ is achieved.\\
(III) If $m>M(Q_{1})$, then
\begin{equation*}
E(u)\geq E_{min}(m) \quad \text{and} \quad E_{min}(m)<0.
\end{equation*}
In this case, $E_{min}(m)$ is achieved.
\end{proposition}

\begin{proposition}\label{p3.6}
(\cite{KOPV2017}) For the variational problem (\ref{1.17}), one has that:\\
(I) $E_{min}^{V}(m)=\infty$ when $0<m<\frac{4}{3\sqrt{3}}M(Q_{1})$.\\
(II) $0<E_{min}^{V}(m)<\infty$ when $\frac{4}{3\sqrt{3}}M(Q_{1})\leq m<M(Q_{1})$.\\
(III) $E_{min}^{V}(m)=E_{min}(m)$ when $M(Q_{1})\leq m<\infty$.\\
(IV) For $m\geq \frac{4}{3\sqrt{3}}M(Q_{1})$, the infimum $E_{min}^{V}(m)$ is achieved and is both strictly decreasing and lower semicontinuous as a function of $m$.
\end{proposition}

\begin{proposition}\label{p3.7}
(\cite{KOPV2017}) Let $m\geq \frac{4}{3\sqrt{3}}M(Q_{1})$ and $u\in H_{x}^{1}(\mathbb{R}^{3})\backslash\{0\}$ obey
\begin{equation*}
M(u)=m,\;\;\; E(u)=E_{min}^{V}(m), \;\;\; V(u)=0.
\end{equation*}
Then either $u(x)=e^{i\theta}R_{\omega}(x+x_{0})$ or $u(x)=e^{i\theta}P_{\omega}(x+x_{0})$ for some $\theta\in [0, 2\pi)$, some $x_{0}\in \mathbb{R}^{3}$ and some $0<\omega<\frac{3}{16}$ that obeys $\beta(\omega)>\frac{1}{3}$. Furthermore, there exist $\delta>0$ so that no ground state has mass $<\frac{4}{3\sqrt{3}}M(Q_{1})+\delta$, while $u$ cannot be a rescaled soliton when $m>M(Q_{1})-\delta$.
\end{proposition}

Furthermore, we give the following some new results.

\begin{theorem}\label{t3.8}
Let $0<\alpha<\infty$ and put
\begin{equation*}
R_{\alpha}=\sqrt{\frac{1+\alpha}{4\alpha}}Q_{\alpha}\Big(\frac{3[1+\alpha]}{4\sqrt{3\alpha}}x\Big),
\end{equation*}
where $Q_{\alpha}$ as in Proposition \ref{p3.1}. Then one has that
\begin{equation*}
\beta(Q_{\alpha})=\alpha,\quad V(Q_{\alpha})=0,\quad \beta(R_{\alpha})=\frac{1}{3},\quad V(R_{\alpha})=0.
\end{equation*}
Moreover, both $Q_{\alpha}$ and $R_{\alpha}$ are positive minimizers of the variational problem (\ref{1.9}). In addition, $Q_{\alpha}$ satisfies
\begin{equation*}
-\Delta Q_{\alpha}+Q_{\alpha}^{5}-Q_{\alpha}^{3}+\omega Q_{\alpha}=0
\end{equation*}
with
\begin{equation*}
\omega=\frac{1+\alpha}{3}\frac{\lvert\lvert \nabla Q_{\alpha}\lvert\lvert_{L^{2}}^{2}}{\lvert \lvert Q_{\alpha}\lvert \lvert_{L^{2}}^{2}}=\frac{16\alpha}{9(1+\alpha)}\frac{\lvert\lvert \nabla R_{\alpha}\lvert\lvert_{L^{2}}^{2}}{\lvert \lvert R_{\alpha}\lvert \lvert_{L^{2}}^{2}}.
\end{equation*}
Accordingly $Q_{\alpha}=P_{\omega}$ (up to translations) is a ground state and $R_{\alpha}=R_{\omega}$ (up to translations) is a rescaled soliton.
\end{theorem}
\begin{proof}
From Proposition \ref{p3.1}, $Q_{\alpha}$ satisfies $\beta(Q_{\alpha})=\alpha$ and
\begin{equation}\label{3.8a}
-\Delta Q_{\alpha}+Q_{\alpha}^{5}-Q_{\alpha}^{3}+\omega Q_{\alpha}=0
\end{equation}
with some $\omega\in (0,\frac{3}{16})$. By (\ref{1.16}) and (\ref{3.8a}), it follows that $V(Q_{\alpha})=0$. Thus
\begin{equation}\label{3.8b}
\int \lvert Q_{\alpha}\lvert^{4}dx=\frac{4(\alpha+1)}{3}\int \lvert \nabla Q_{\alpha}\lvert^{2}dx,\quad \int \lvert Q_{\alpha}\lvert^{6}dx=\alpha\int \lvert \nabla Q_{\alpha}\lvert^{2}dx.
\end{equation}
By (\ref{1.8}), it follows that
\begin{equation}\label{3.8c}
F_{\alpha}(Q_{\alpha})=\frac{3\alpha^{\frac{\alpha}{2(1+\alpha)}}}{4(1+\alpha)}\lvert\lvert Q_{\alpha}\lvert\lvert_{L^{2}}\lvert\lvert \nabla Q_{\alpha}\lvert\lvert_{L^{2}}^{\frac{1-\alpha}{1+\alpha}}.
\end{equation}
According to Proposition \ref{p3.1},
\begin{equation}\label{3.8d}
F_{\alpha}(Q_{\alpha})=C_{\alpha}^{-1}=\mathop{\mathrm{min}}_{\{u\in H^1(\mathbb{R}^3)\backslash\{0\}\}} F_{\alpha}(u).
\end{equation}
Therefore $Q_{\alpha}$ is a positive minimizer of the variational problem (\ref{1.9}).

Since
\begin{equation}\label{3.8e}
R_{\alpha}=\sqrt{\frac{1+\alpha}{4\alpha}}Q_{\alpha}\Big(\frac{3[1+\alpha]}{4\sqrt{3\alpha}}x\Big),
\end{equation}
one implies that
\begin{equation}\label{3.8f}
\beta(R_{\alpha})=\frac{1}{3},\quad V(R_{\alpha})=0.
\end{equation}
Thus one has that
\begin{equation}\label{3.8g}
\int \lvert R_{\alpha}\lvert^{4}dx=\frac{16}{9} \int \lvert \nabla R_{\alpha}\lvert^{2}dx,\quad \int \lvert R_{\alpha}\lvert^{6}dx=\frac{1}{3} \int \lvert \nabla R_{\alpha}\lvert^{2}dx.
\end{equation}
By (\ref{1.8}),
\begin{equation}\label{3.8h}
F_{\alpha}(R_{\alpha})=\frac{9}{16}\big(\frac{1}{3}\big)^{\frac{\alpha}{2(1+\alpha)}}\big(\int \lvert R_{\alpha}\lvert^{2}dx\big)^{\frac{1}{2}}\big(\int \lvert \nabla R_{\alpha}\lvert^{2}dx\big)^{\frac{1-\alpha}{2(1+\alpha)}}.
\end{equation}
From (\ref{3.8e}) and (\ref{3.8h}), one yields that
\begin{equation}\label{3.8i}
F_{\alpha}(R_{\alpha})=\frac{3\alpha^{\frac{\alpha}{2(1+\alpha)}}}{4(1+\alpha)}\lvert\lvert Q_{\alpha}\lvert\lvert_{L^{2}}\lvert\lvert \nabla Q_{\alpha}\lvert\lvert_{L^{2}}^{\frac{1-\alpha}{1+\alpha}}.
\end{equation}
Thus
\begin{equation}\label{3.8j}
F_{\alpha}(R_{\alpha})=C_{\alpha}^{-1}=\mathop{\mathrm{min}}_{\{u\in H^1(\mathbb{R}^3)\backslash\{0\}\}} F_{\alpha}(u).
\end{equation}
Therefore $R_{\alpha}$ is also a positive minimizer of the variational problem (\ref{1.9}).

Since $Q_{\alpha}$ is a positive minimizer of $C_{\alpha}^{-1}$, it must satisfy the corresponding Euler-Lagrange equation:
\begin{equation}\label{3.8k}
\frac{d}{d\varepsilon}F_{\alpha}(Q_{\alpha}+\varepsilon \phi)=0\quad \text{for all}\quad  \phi\in \mathcal{C}_{c}^{\infty}(\mathbb{R}^{3}).
\end{equation}
Thus, direct computation shows that $Q_{\alpha}$ is a distribution solution to the following equation:
\begin{equation}\label{3.8l}
-\Delta Q_{\alpha}+Q_{\alpha}^{5}-Q_{\alpha}^{3}+\omega Q_{\alpha}=0\quad \text{with}\quad \omega=\frac{1+\alpha}{3}\frac{\lvert\lvert \nabla Q_{\alpha}\lvert\lvert_{L^{2}}^{2}}{\lvert \lvert Q_{\alpha}\lvert \lvert_{L^{2}}^{2}}.
\end{equation}

Since $R_{\alpha}$ is a positive minimizer of $C_{\alpha}^{-1}$, it must satisfy the corresponding Euler-Lagrange equation:
\begin{equation}\label{3.8m}
\frac{d}{d\varepsilon}F_{\alpha}(R_{\alpha}+\varepsilon \phi)=0\quad \text{for all}\quad  \phi\in \mathcal{C}_{c}^{\infty}(\mathbb{R}^{3}).
\end{equation}
Thus, direct computation shows that $R_{\alpha}$ is a distribution solution to the following equation:
\begin{equation}\label{3.8n}
-\Delta R_{\alpha}+\alpha\frac{\int\lvert \nabla R_{\alpha}\lvert^{2}dx }{\int R_{\alpha}^{6}dx}R_{\alpha}^{5}-\frac{4(1+\alpha)}{3}\frac{\int\lvert \nabla R_{\alpha}\lvert^{2}dx}{\int R_{\alpha}^{4}dx}R_{\alpha}^{3}+\frac{1+\alpha}{3}\frac{\int\lvert \nabla R_{\alpha}\lvert^{2}dx}{\int R_{\alpha}^{2}dx}R_{\alpha}=0.
\end{equation}
Note that
\begin{equation}\label{3.8o}
R_{\alpha}=\sqrt{\frac{1+\alpha}{4\alpha}}Q_{\alpha}\Big(\frac{3[1+\alpha]}{4\sqrt{3\alpha}}x\Big).
\end{equation}
(\ref{3.8o}) and (\ref{3.8n}) yield that $Q_{\alpha}$ satisfies
\begin{equation}\label{3.8p}
-\Delta Q_{\alpha}+Q_{\alpha}^{5}-Q_{\alpha}^{3}+\omega Q_{\alpha}=0\quad \text{with}\quad \omega=\frac{16\alpha}{9(1+\alpha)}\frac{\lvert\lvert \nabla R_{\alpha}\lvert\lvert_{L^{2}}^{2}}{\lvert \lvert R_{\alpha}\lvert \lvert_{L^{2}}^{2}}.
\end{equation}
(\ref{3.8p}) shows that $Q_{\alpha}$ is a ground state. By Proposition \ref{p2.1}, uniqueness of ground state yields that
\begin{equation}\label{3.8q}
Q_{\alpha}=P_{\omega}\;\;\;\text{up to translations},
\end{equation}
which is a ground state. Since $\alpha=\beta(Q_{\alpha})$, one gets that
\begin{equation}\label{3.8r}
\alpha=\beta(P_{\omega})=\beta(\omega).
\end{equation}
Thus $R_{\alpha}=R_{\omega}$ (up to translations) is a rescaled soliton, where $R_{\omega}$ as in Proposition \ref{p2.3}.
\end{proof}

\begin{theorem}\label{t3.9}
Let $0<\alpha<\infty$. Then $C_{\alpha}^{-1}$ identifies a unique $Q_{\alpha}(x)>0$ satisfying
\begin{equation*}
-\Delta Q_{\alpha}+Q_{\alpha}^{5}-Q_{\alpha}^{3}+\omega Q_{\alpha}=0
\end{equation*}
with
\begin{equation*}
\omega=\frac{1+\alpha}{3}\frac{\lvert\lvert \nabla Q_{\alpha}\lvert\lvert_{L^{2}}^{2}}{\lvert \lvert Q_{\alpha}\lvert \lvert_{L^{2}}^{2}}.%=\frac{16\alpha}{9(1+\alpha)}\frac{\lvert\lvert \nabla R_{\alpha}\lvert\lvert_{L^{2}}^{2}}{\lvert \lvert R_{\alpha}\lvert \lvert_{L^{2}}^{2}}.
\end{equation*}
In addition, $\omega$ is completely identified by $\alpha$ and
\begin{equation*}
Q_{\alpha}=P_{\omega}\;\;\;\text{up to translations}.
\end{equation*}
%In addition, $Q_{\alpha}$ is a positive minimizer of $C_{\alpha}^{-1}$.
\end{theorem}

\begin{proof}
For $0<\alpha<\infty$, from Proposition \ref{p3.1} $C_{\alpha}$ is completely identified by $\alpha$. In addition,
\begin{equation}\label{3.9a}
C_{\alpha}^{-1}=\frac{3\alpha^{\frac{\alpha}{2(1+\alpha)}}}{4(1+\alpha)}\lvert\lvert Q_{\alpha}\lvert\lvert_{L^{2}}\lvert\lvert \nabla Q_{\alpha}\lvert\lvert_{L^{2}}^{\frac{1-\alpha}{1+\alpha}}.
\end{equation}
Note that here $Q_{\alpha}$ is only a positive minimizer of $C_{\alpha}^{-1}$, which is not identified by $\alpha$. But by Proposition \ref{p3.3},
\begin{equation}\label{3.9b}
\alpha \mapsto \lvert\lvert \nabla Q_{\alpha}\lvert\lvert_{L^{2}}^{2}\quad \text{is strictly increasing on}\; (0,\infty).
\end{equation}
Hence (\ref{3.9b}) implies that $\lvert\lvert \nabla Q_{\alpha}\lvert\lvert_{L^{2}}^{2}$ is completely identified by $\alpha$. Thus we set
\begin{equation}\label{3.9c}
\lvert\lvert \nabla Q_{\alpha}\lvert\lvert_{L^{2}}^{2}=h(\alpha).
\end{equation}
In terms of Theorem \ref{t3.8}, $Q_{\alpha}$ satisfies
\begin{equation}\label{3.9d}
-\Delta Q_{\alpha}+Q_{\alpha}^{5}-Q_{\alpha}^{3}+\omega Q_{\alpha}=0
\end{equation}
with
\begin{equation}\label{3.9e}
\omega=\frac{1+\alpha}{3}\frac{\lvert\lvert \nabla Q_{\alpha}\lvert\lvert_{L^{2}}^{2}}{\lvert \lvert Q_{\alpha}\lvert \lvert_{L^{2}}^{2}}.
\end{equation}
By (\ref{3.9a}) and (\ref{3.9c}), one gets that
\begin{equation}\label{3.9f}
\lvert\lvert Q_{\alpha}\lvert\lvert_{L^{2}}^{2}=\Big[\frac{4(1+\alpha)}{3\alpha^{\frac{\alpha}{2(1+\alpha)}}} \Big]^{2}\Big[\frac{1}{C_{\alpha}(h(\alpha))^{\frac{1-\alpha}{2(1+\alpha)}}}\Big]^{2}:=g(\alpha).
\end{equation}
Thus (\ref{3.9f}) shows that $\lvert\lvert  Q_{\alpha}\lvert\lvert_{L^{2}}^{2}$ is completely identified by $\alpha$, although $ Q_{\alpha}$ is not identified by $\alpha$. By (\ref{3.9e}), (\ref{3.9c}) and (\ref{3.9f}), one gets that
\begin{equation}\label{3.9g}
\omega=\frac{1+\alpha}{3}\frac{\lvert\lvert \nabla Q_{\alpha}\lvert\lvert_{L^{2}}^{2}}{\lvert \lvert Q_{\alpha}\lvert \lvert_{L^{2}}^{2}}=\frac{1+\alpha}{3}\frac{h(\alpha)}{g(\alpha)}:=f(\alpha).
\end{equation}
It follows that
\begin{equation}\label{3.9h}
\omega\;\;\;\text{is completely identified by}\;\; \alpha.
\end{equation}
According to Theorem \ref{t3.8},
\begin{equation}\label{3.9i}
Q_{\alpha}=P_{\omega}\;\;\;\;\text{up to translations}.
\end{equation}
Then one gets that
\begin{equation}\label{3.9j}
Q_{\alpha}\;\;\;\text{is uniquely identified by}\;\; \omega.
\end{equation}
(\ref{3.9h}) and (\ref{3.9j}) yield that $Q_{\alpha}$ is uniquely identified by $\alpha$. Therefore for $0<\alpha<\infty$, $C_{\alpha}^{-1}$ identifies a unique $Q_{\alpha}(x)>0$ satisfying
\begin{equation}\label{3.9k}
-\Delta Q_{\alpha}+Q_{\alpha}^{5}-Q_{\alpha}^{3}+\omega Q_{\alpha}=0
\end{equation}
with
\begin{equation}\label{3.9l}
\omega=\frac{1+\alpha}{3}\frac{\lvert\lvert \nabla Q_{\alpha}\lvert\lvert_{L^{2}}^{2}}{\lvert \lvert Q_{\alpha}\lvert \lvert_{L^{2}}^{2}}\;\;\;\text{and}\;\;\;Q_{\alpha}=P_{\omega}\;\;\;\;\text{up to translations}.
\end{equation}
In terms of Theorem \ref{t3.8}, $Q_{\alpha}$ is a positive minimizer of $C_{\alpha}^{-1}$.
\end{proof}

\begin{theorem}\label{t3.10}
Let $0<\alpha\leq 1$ and $Q_{\alpha}$ as in Proposition \ref{p3.1}. Then
\begin{equation*}
\alpha\mapsto \frac{\lvert\lvert \nabla Q_{\alpha}\lvert\lvert_{L^{2}}^{2}}{\lvert \lvert Q_{\alpha}\lvert \lvert_{L^{2}}^{2}} \quad  \text{is strictly increasing}.
\end{equation*}
\end{theorem}
\begin{proof}
Assume that $\alpha$ and $\nu$ satisfy
\begin{equation}\label{3.10a}
0<\nu<\alpha\leq 1.
\end{equation}
By Theorem \ref{t3.8},  $Q_{\alpha}$ is a positive minimizer of $C_{\alpha}^{-1}$, and $Q_{\nu}$ is a positive minimizer of $C_{\nu}^{-1}$. Thus one has that
\begin{equation}\label{3.10b}
F_{\alpha}(Q_{\nu})> F_{\alpha}(Q_{\alpha})=C_{\alpha}^{-1},\quad F_{\nu}(Q_{\alpha})> F_{\nu}(Q_{\nu})=C_{\nu}^{-1}.
\end{equation}
By Proposition \ref{p3.1}, one has that
\begin{equation}\label{3.10c}
V(Q_{\alpha})=0=V(Q_{\nu}),\quad \beta(Q_{\alpha})=\alpha,\quad \beta(Q_{\nu})=\nu.
\end{equation}
It follows  that
\begin{equation}\label{3.10d}
F_{\alpha}(Q_{\nu})=\frac{3\nu^{\frac{\alpha}{2(1+\alpha)}}}{4(1+\nu)}\lvert\lvert Q_{\nu}\lvert\lvert_{L^{2}}\big( \int \lvert \nabla Q_{\nu}\lvert^{2}dx\big)^{\frac{1-\alpha}{2(1+\alpha)}},
%\lvert\lvert \nabla Q_{\alpha}\lvert\lvert_{L^{2}}^{\frac{1-\alpha}{1+\alpha}}.
\end{equation}
\begin{equation}\label{3.10e}
F_{\nu}(Q_{\alpha})=\frac{3\alpha^{\frac{\nu}{2(1+\nu)}}}{4(1+\alpha)}\lvert\lvert Q_{\alpha}\lvert\lvert_{L^{2}}\big( \int \lvert \nabla Q_{\alpha}\lvert^{2}dx\big)^{\frac{1-\nu}{2(1+\nu)}},
\end{equation}
\begin{equation}\label{3.10f}
F_{\alpha}(Q_{\alpha})=\frac{3\alpha^{\frac{\alpha}{2(1+\alpha)}}}{4(1+\alpha)}\lvert\lvert Q_{\alpha}\lvert\lvert_{L^{2}}\big( \int \lvert \nabla Q_{\alpha}\lvert^{2}dx\big)^{\frac{1-\alpha}{2(1+\alpha)}},
\end{equation}
\begin{equation}\label{3.10g}
F_{\nu}(Q_{\nu})=\frac{3\nu^{\frac{\nu}{2(1+\nu)}}}{4(1+\nu)}\lvert\lvert Q_{\nu}\lvert\lvert_{L^{2}}\big( \int \lvert \nabla Q_{\nu}\lvert^{2}dx\big)^{\frac{1-\nu}{2(1+\nu)}}.
\end{equation}
By (\ref{3.10b}), (\ref{3.10d}), (\ref{3.10e}), (\ref{3.10f}) and (\ref{3.10g}), one has that
\begin{equation}\label{3.10h}
\frac{1+\nu}{1+\alpha}\big(\frac{\nu}{\alpha}\big)^{-\frac{\nu}{2(1+\nu)}}\Big(\frac{\lvert\lvert \nabla Q_{\alpha}\lvert\lvert_{L^{2}}}{\lvert\lvert\nabla Q_{\nu}\lvert\lvert_{L^{2}}} \Big)^{\frac{1-\nu}{1+\nu}}> \frac{\lvert\lvert  Q_{\nu}\lvert\lvert_{L^{2}}}{\lvert\lvert Q_{\alpha}\lvert\lvert_{L^{2}}},
\end{equation}
\begin{equation}\label{3.10i}
\frac{\lvert\lvert  Q_{\nu}\lvert\lvert_{L^{2}}}{\lvert\lvert Q_{\alpha}\lvert\lvert_{L^{2}}}> \frac{1+\nu}{1+\alpha}\big(\frac{\nu}{\alpha}\big)^{-\frac{\alpha}{2(1+\alpha)}}\Big(\frac{\lvert\lvert \nabla Q_{\alpha}\lvert\lvert_{L^{2}}}{\lvert\lvert\nabla Q_{\nu}\lvert\lvert_{L^{2}}} \Big)^{\frac{1-\alpha}{1+\alpha}}.
\end{equation}
From (\ref{3.10h}) and (\ref{3.10i}), one gets that
\begin{equation}\label{3.10j}
\frac{\lvert\lvert \nabla Q_{\alpha}\lvert\lvert_{L^{2}}^{2}}{\lvert\lvert \nabla Q_{\nu}\lvert\lvert_{L^{2}}^{2}}>\big(\frac{\alpha}{\nu} \big)^{\frac{1}{2}}.
\end{equation}
By (\ref{3.10a}), (\ref{3.10h}) and (\ref{3.10j}), one gets that
\begin{equation}\label{3.10k}
\frac{\lvert\lvert Q_{\nu}\lvert\lvert_{L^{2}}^{2}}{\lvert\lvert  Q_{\alpha}\lvert\lvert_{L^{2}}^{2}}>\frac{(1+\nu)^{2}\alpha^{\frac{1}{2}}}{(1+\alpha)^{2}\nu^{\frac{1}{2}}}.
\end{equation}
From (\ref{3.10j}) and (\ref{3.10k}), one has that
\begin{equation}\label{3.10l}
\frac{\lvert\lvert \nabla Q_{\alpha}\lvert\lvert_{L^{2}}^{2}}{\lvert\lvert \nabla Q_{\nu}\lvert\lvert_{L^{2}}^{2}}\cdot\frac{\lvert\lvert Q_{\nu}\lvert\lvert_{L^{2}}^{2}}{\lvert\lvert  Q_{\alpha}\lvert\lvert_{L^{2}}^{2}}>\frac{(1+\nu)^{2}\alpha}{(1+\alpha)^{2}\nu}.
\end{equation}
Since
\begin{equation}\label{3.10m}
\frac{d}{d\alpha}[(1+\alpha)^{2}\alpha^{-1}]=\frac{\alpha^{2}-1}{\alpha^{2}},
\end{equation}
it follows that
\begin{equation}\label{3.10n}
(1+\alpha)^{2}\alpha^{-1}\;\;\;\text{is strictly decreasing on}\;\;\; (0,1].
\end{equation}
Then (\ref{3.10l}) and (\ref{3.10n}) imply that
\begin{equation}\label{3.10o}
\alpha\mapsto \frac{\lvert\lvert \nabla Q_{\alpha}\lvert\lvert_{L^{2}}^{2}}{\lvert\lvert Q_{\alpha}\lvert\lvert_{L^{2}}^{2}} \quad \text{is strictly increasing on} \; (0,1].
\end{equation}
\end{proof}

\begin{theorem}\label{t3.11}
Let $0<\alpha<\infty$ and $R_{\alpha}$ as in Theorem \ref{t3.8}. Then one has that
\begin{equation*}
\alpha\mapsto \lvert\lvert \nabla R_{\alpha}\lvert\lvert_{L^{2}}^{2} \quad  \text{is strictly increasing on}\; (0,\infty),
\end{equation*}
\begin{equation*}
\alpha\mapsto M(R_{\alpha})=\lvert\lvert R_{\alpha}\lvert\lvert_{L^{2}}^{2} \quad  \text{is strictly decreasing on}\;(0,1].
\end{equation*}
\begin{equation*}
\alpha\mapsto M(R_{\alpha})=\lvert\lvert R_{\alpha}\lvert\lvert_{L^{2}}^{2} \quad  \text{is strictly increasing on}\;(1,\infty).
\end{equation*}
\end{theorem}
\begin{proof}
Assume that $0<\alpha, \nu <\infty$ and $\alpha>\nu$. According to Proposition \ref{p3.1} and Theorem \ref{t3.8}, there exists $Q_{\alpha}(x)$ satisfies
\begin{equation}\label{3.11a}
C_{\alpha}^{-1}=\mathop{\mathrm{min}}_{\{u\in H^1(\mathbb{R}^3)\backslash\{0\}\}} F_{\alpha}(u),
\end{equation}
and there exists $Q_{\nu}(x)$ satisfies
\begin{equation}\label{3.11b}
C_{\nu}^{-1}=\mathop{\mathrm{min}}_{\{u\in H^1(\mathbb{R}^3)\backslash\{0\}\}} F_{\nu}(u).
\end{equation}
Let $R_{\alpha}$ as in Theorem \ref{t3.8} and $R_{\nu}=\sqrt{\frac{1+\nu}{4\nu}}Q_{\nu}\Big(\frac{3[1+\nu]}{4\sqrt{3\nu}}x\Big)$. Then by Theorem \ref{t3.8}, $R_{\alpha}(x)$ satisfies (\ref{3.11a}) and $R_{\nu}(x)$ satisfies (\ref{3.11b}), that is
\begin{equation}\label{3.11c}
\mathop{\mathrm{min}}_{\{u\in H^1(\mathbb{R}^3)\backslash\{0\}\}} F_{\alpha}(u)=F_{\alpha}(R_{\alpha}), \quad \mathop{\mathrm{min}}_{\{u\in H^1(\mathbb{R}^3)\backslash\{0\}\}} F_{\nu}(u)=F_{\nu}(R_{\nu}).
\end{equation}
Since $\alpha\neq \nu$, it follows that
\begin{equation}\label{3.11d}
F_{\alpha}(R_{\nu})>F_{\alpha}(R_{\alpha}), \quad F_{\nu}(R_{\alpha})>F_{\nu}(R_{\nu}).
\end{equation}
By (\ref{3.8h}), one has that
\begin{equation}\label{3.11e}
F_{\nu}(R_{\nu})=\frac{9}{16}\big(\frac{1}{3}\big)^{\frac{\nu}{2(1+\nu)}}\big(\int \lvert R_{\nu}\lvert^{2}dx\big)^{\frac{1}{2}}\big(\int \lvert \nabla R_{\nu}\lvert^{2}dx\big)^{\frac{1-\nu}{2(1+\nu)}}.
\end{equation}
By (\ref{3.8h}), (\ref{3.11e}) and (\ref{3.11d}), one implies that
\begin{equation}\label{3.11f}
\lvert \lvert R_{\nu}\lvert \lvert_{L^{2}}\lvert \lvert \nabla R_{\nu}\lvert \lvert_{L^{2}}^{\frac{1-\alpha}{1+\alpha}}> \lvert \lvert R_{\alpha}\lvert \lvert_{L^{2}}\lvert \lvert \nabla R_{\alpha}\lvert \lvert_{L^{2}}^{\frac{1-\alpha}{1+\alpha}},
\end{equation}
\begin{equation}\label{3.11g}
\lvert \lvert R_{\alpha}\lvert \lvert_{L^{2}}\lvert \lvert \nabla R_{\alpha}\lvert \lvert_{L^{2}}^{\frac{1-\nu}{1+\nu}}> \lvert \lvert R_{\nu}\lvert \lvert_{L^{2}}\lvert \lvert \nabla R_{\nu}\lvert \lvert_{L^{2}}^{\frac{1-\nu}{1+\nu}}.
\end{equation}
(\ref{3.11f}) and (\ref{3.11g}) yield that
\begin{equation}\label{3.11h}
\Big( \frac{\lvert \lvert \nabla R_{\alpha}\lvert \lvert_{L^{2}}}{\lvert \lvert \nabla R_{\nu}\lvert \lvert_{L^{2}}}\Big)^{\frac{1-\nu}{1+\nu}}>\frac{\lvert \lvert R_{\nu}\lvert \lvert_{L^{2}}}{\lvert \lvert R_{\alpha}\lvert \lvert_{L^{2}}}>\Big( \frac{\lvert \lvert \nabla R_{\alpha}\lvert \lvert_{L^{2}}}{\lvert \lvert \nabla R_{\nu}\lvert \lvert_{L^{2}}}\Big)^{\frac{1-\alpha}{1+\alpha}}.
\end{equation}
It follows that
\begin{equation}\label{3.11i}
\Big( \frac{\lvert \lvert \nabla R_{\alpha}\lvert \lvert_{L^{2}}}{\lvert \lvert \nabla R_{\nu}\lvert \lvert_{L^{2}}}\Big)^{\frac{2(\alpha-\nu)}{(1+\nu)(1+\alpha)}}>1.
\end{equation}
From $\alpha>\nu$, one gets that
\begin{equation}\label{3.11j}
\lvert \lvert \nabla R_{\alpha}\lvert \lvert_{L^{2}}>\lvert \lvert \nabla R_{\nu}\lvert \lvert_{L^{2}}.
\end{equation}
Therefore we conclude that %for $\alpha\in (0,\infty)$,
\begin{equation}\label{3.11k}
\alpha\mapsto \lvert\lvert \nabla R_{\alpha}\lvert\lvert_{L^{2}}^{2} \quad  \text{is strictly increasing on}\;(0,\infty).
\end{equation}

When $\frac{1-\alpha}{1+\alpha}\geq 0$, that is, $0<\alpha \leq 1$, by (\ref{3.11f}) one gets that
\begin{equation}\label{3.11l}
 \frac{\lvert \lvert R_{\nu}\lvert \lvert_{L^{2}}}{\lvert \lvert R_{\alpha}\lvert \lvert_{L^{2}}}>\Big(\frac{\lvert \lvert \nabla R_{\alpha}\lvert \lvert_{L^{2}}}{\lvert \lvert \nabla R_{\nu}\lvert \lvert_{L^{2}}}\Big)^{\frac{1-\alpha}{1+\alpha}}.
\end{equation}
Thus for $0<\nu<\alpha\leq 1$, from  (\ref{3.11j}) and (\ref{3.11l}), one has that
\begin{equation}\label{3.11m}
\lvert \lvert R_{\nu}\lvert \lvert _{L^{2}}^{2}>\lvert \lvert R_{\alpha}\lvert \lvert _{L^{2}}^{2}.
\end{equation}
Therefore
\begin{equation}\label{3.11n}
\alpha\mapsto M(R_{\alpha})=\lvert\lvert R_{\alpha}\lvert\lvert_{L^{2}}^{2} \quad  \text{is strictly decreasing on}\;(0,1].
\end{equation}

When $\frac{1-\nu}{1+\nu}<0$, that is, $\nu>1$, by (\ref{3.11g}) one gets that
\begin{equation}\label{3.11o}
\frac{\lvert \lvert R_{\nu}\lvert \lvert_{L^{2}}}{\lvert \lvert R_{\alpha}\lvert \lvert_{L^{2}}}<\Big( \frac{\lvert \lvert \nabla R_{\alpha}\lvert \lvert_{L^{2}}}{\lvert \lvert \nabla R_{\nu}\lvert \lvert_{L^{2}}}\Big)^{\frac{1-\nu}{1+\nu}}
\end{equation}
Thus for $1<\nu<\alpha$, from  (\ref{3.11j}) and (\ref{3.11o}), one has that
\begin{equation}\label{3.11p}
\lvert \lvert R_{\nu}\lvert \lvert _{L^{2}}^{2}<\lvert \lvert R_{\alpha}\lvert \lvert _{L^{2}}^{2}.
\end{equation}
Therefore
\begin{equation}\label{3.11q}
\alpha\mapsto M(R_{\alpha})=\lvert\lvert R_{\alpha}\lvert\lvert_{L^{2}}^{2} \quad  \text{is strictly increasing on}\;(1,\infty).
\end{equation}
\end{proof}

\section{The case for $\alpha\geq 1$}\label{sec4}

\begin{theorem}\label{t4.1}
Let $m\geq M(Q_{1})$, where $Q_{1}$ as in Proposition \ref{p3.2}. Then $m$ identifies a unique $\omega\in (0,\frac{3}{16})$ satisfying $m=M(P_{\omega})$.
\end{theorem}
\begin{proof}
According to Proposition \ref{p3.3},
\begin{equation}\label{4.1a}
\alpha \mapsto M(Q_{\alpha})=\lvert\lvert Q_{\alpha}\lvert\lvert_{L^{2}}^{2} \quad \text{is strictly increasing on}\;  [1,\infty).
\end{equation}
Then for $m\geq M(Q_{1})$,
\begin{equation}\label{4.1b}
m\;\;\;\text{identifies a unique}\;\;\; \alpha\geq 1\;\;\;\text{with}\;\;\;m=M(Q_{\alpha}).
\end{equation}
In terms of Theorem \ref{t3.9}, $\alpha$ identifies a unique $\omega$ satisfying
\begin{equation}\label{4.1c}
\omega=\frac{1+\alpha}{3}\frac{\lvert\lvert \nabla Q_{\alpha}\lvert\lvert_{L^{2}}^{2}}{\lvert \lvert Q_{\alpha}\lvert \lvert_{L^{2}}^{2}}:=f(\alpha)\;\;\;\text{with}\;\;\; Q_{\alpha}=P_{\omega}\;\;\;\text{up to translations}.
\end{equation}
By (\ref{4.1b}) and (\ref{4.1c}), one gets that
\begin{equation}\label{4.1d}
m\;\;\;\text{identifies a unique}\;\;\;\omega\in (0,\frac{3}{16})\;\;\text{satisfying}\;\; m=M(P_{\omega}).
\end{equation}
\end{proof}

\begin{theorem}\label{t4.2}
Let $m\geq M(Q_{1})$, where $Q_{1}$ as in Proposition \ref{p3.2}. Then $E_{min}^{V}(m)$ possesses a unique positive minimizer $P_{\omega}$ up to translations. %Moreover $\beta(Q_{1})=1$.
\end{theorem}
\begin{proof}
Since $m\geq M(Q_{1})$, from Proposition \ref{p3.6}, one has that
\begin{equation}\label{4.2a}
E_{min}^{V}(m)=E_{min}(m),
\end{equation}
and $E_{min}^{V}(m)$ possesses a positive minimizer $Q$. Then $Q$ satisfies
\begin{equation}\label{4.2b}
M(Q)=m\geq M(Q_{1}).
\end{equation}
Moreover, $Q$ satisfies the Euler-Lagrange equation with the Lagrange multiplier $\omega$
\begin{equation}\label{4.2c}
-\Delta Q+Q^{5}-Q^{3}+\omega Q=0.
\end{equation}
It follows that
\begin{equation}\label{4.2d}
\omega\in (0,\frac{3}{16})\;\;\;\text{and}\;\;\;Q=P_{\omega}\;\;\;\text{up to transalations}.
\end{equation}
From (\ref{4.2b}) and (\ref{4.2d}),
\begin{equation}\label{4.2e}
M(P_{\omega})=m\geq M(Q_{1}).
\end{equation}
Since $m\geq M(Q_{1})$, by Theorem \ref{t4.1} one has that
\begin{equation}\label{4.2f}
m\;\;\;\text{identifies a unique}\;\;\;\omega\in (0,\frac{3}{16})\;\;\text{satisfying}\;\; m=M(P_{\omega}).
\end{equation}
It follows that
\begin{equation}\label{4.2g}
m\;\;\;\text{identifies a unique ground state}\;\;P_{\omega}\;\;\text{satisfying}\;\; m=M(P_{\omega}).
\end{equation}
By (\ref{4.2d}) and (\ref{4.2g}), one yields that
\begin{equation}\label{4.2h}
m\;\;\;\text{identifies a unique positive minimizer}\;\;Q \;\;\text{of}\;\;E_{min}^{V}(m).
\end{equation}
Therefore $E_{min}^{V}(m)$ possesses a unique positive minimizer $P_{\omega}$ for $m\geq M(Q_{1})$.
\end{proof}

\begin{theorem}\label{t4.3}
There exists $\omega^{\ast}\in (0,\;\frac{3}{16})$ such that $\beta(\omega^{\ast})=1$. Moreover, $M(P_{\omega})$ is strictly increasing on $\omega\in [\omega^{\ast},\;\frac{3}{16})$.
\end{theorem}
\begin{proof}
According to Proposition \ref{p2.2}, $\beta(\omega)$ is a continuous function from $(0,\;\frac{3}{16})$ to $(0,\; \infty)$. Therefore there exists $\omega^{\ast}\in (0,\;\frac{3}{16})$ such that $\beta(\omega^{\ast})=1$. In terms of Theorem \ref{t3.8},
\begin{equation}\label{4.3a}
P_{\omega^{\ast}}=Q_{1}\;\;\;\text{up to translations}.
\end{equation}
It follows that
\begin{equation}\label{4.3b}
M(P_{\omega^{\ast}})=M(Q_{1}).
\end{equation}
By Proposition \ref{p2.2},
\begin{equation}\label{4.3c}
M(\omega)=M(P_{\omega})\rightarrow\infty\;\;\;\text{as}\;\;\;\omega\rightarrow\frac{3}{16}.
\end{equation}
From (\ref{4.3b}) and (\ref{4.3c}),  Theorem \ref{t4.2} establishes a one-to-one mapping from $[M(P_{\omega^{\ast}}),\; \infty)$ to $[\omega^{\ast},\; \frac{3}{16})$. It follows that
\begin{equation}\label{4.3d}
M(P_{\omega})\;\;\; \text{is strictly increasing on} \;\;\;[\omega^{\ast},\; \frac{3}{16}).
\end{equation}
\end{proof}

\begin{proposition}\label{p4.4}
Let $f(z)$ be real analytic on $(a,b)$ and $f(z)\not\equiv 0$ for $z\in(a,b)$. If $z_{0}\in(a,b)$ satisfies $f(z_{0})=0$, then $z_{0}$ is an isolated zero point of $f(z)$, which means that there exists $\delta>0$ such that
\begin{equation*}
f(z)\neq 0 \;\;\; \text{for} \;\;\; z\in(z_{0}-\delta,z_{0}+\delta)\backslash\{z_{0}\}.
\end{equation*}
\end{proposition}
\begin{proof}
Since $f(z_{0})=0$, we assume that
\begin{equation}\label{4.4a}
f(z_0)=f^{\prime}(z_0)=\cdots=f^{(n-1)}(z_0)=0 \;\;\;\text{and}\;\;\; f^{(n)}(z_0)\neq 0,\ n\geq 1.
\end{equation}
It follows that
\begin{equation}\label{4.4b}
f(z)=\frac{f^{(n)}(z_0)}{n!}(z-z_0)^n+\frac{f^{(n+1)}(z_0)}{(n+1)!}(z-z_0)^{n+1}+\cdots=(z-z_0)^n g(z).
\end{equation}
It is clear that $g(z_{0})\neq0$, $g(z)$ is continuous and  analytic on $(a,b)$. Thus there exists $\delta>0$ such that
\begin{equation*}
g(z)\neq0 \;\;\;\text{for} \;\;\;  z\in(z_{0}-\delta,z_0+\delta).
\end{equation*}
From (\ref{4.4b}) we have that
\begin{equation}\label{4.4c}
f(z_{0})=0,\;\;\; f(z)\neq0\;\;\; \text{for} \;\;\; z\in(z_{0}-\delta,z_{0}+\delta)\backslash\{z_{0}\}.
\end{equation}
Therefore $z_{0}$ is an isolated zero point of $f(z)$.
\end{proof}

\begin{theorem}\label{t4.5}
$\beta(\omega)$ is strictly increasing on $[\omega^{\ast},\; \frac{3}{16})$.
\end{theorem}
\begin{proof}
According to Proposition \ref{p2.1}, $M(\omega)$ and  $\beta(\omega)$ are real-analytic on $(0,\; \frac{3}{16})$, and $M(\alpha)$ is real-analytic on $(0,\; \infty)$. From Theorem \ref{t4.3},
\begin{equation}\label{4.5a}
\frac{dM}{d\omega}\geq 0,\;\;\;\omega\in [\omega^{\ast},\; \frac{3}{16}).
\end{equation}
From Proposition \ref{p4.4}, it follows that
\begin{equation}\label{4.5a1}
\frac{dM}{d\omega}>0,\;\;\;\text{almost every}\;\;\;\omega\in [\omega^{\ast},\; \frac{3}{16}).
\end{equation}
By Proposition \ref{p3.3},
\begin{equation}\label{4.5b}
\frac{dM}{d\alpha}\geq 0,\;\;\; \alpha\in [1,\; \infty).
\end{equation}
From Proposition \ref{p4.4}, it follows that
\begin{equation}\label{4.5b1}
\frac{dM}{d\alpha}>0,\;\;\;\text{almost every}\;\;\; \alpha\in [1,\; \infty).
\end{equation}
In terms of Theorem \ref{t3.8},
\begin{equation}\label{4.5c}
\alpha=\beta(Q_{\alpha})=\beta(P_{\omega})=\beta(\omega).
\end{equation}
Then one has that
\begin{equation}\label{4.5d}
\frac{dM}{d\omega}=\frac{dM}{d\alpha}\cdot \frac{d\beta}{d\omega}.
\end{equation}
By (\ref{4.5a1}), (\ref{4.5b1}) and (\ref{4.5d}),
\begin{equation}\label{4.5e}
\frac{d\beta}{d\omega}>0,\;\;\;\text{almost every}\;\;\;\omega\in [\omega^{\ast},\; \frac{3}{16}).
\end{equation}
Thus (\ref{4.5e}) implies that $\beta(\omega)$ is strictly increasing on $[\omega^{\ast},\; \frac{3}{16})$.
\end{proof}

\section{The case for $\frac{1}{3}\leq \alpha<1$}\label{sec5}

\begin{theorem}\label{t5.1}
Let $0<\omega<\frac{3}{16}$ and $P_{\omega}$ be the ground state of (\ref{1.2}) where $Q_{1}$ as in Proposition \ref{p3.2}. If $M(P_{\omega})\leq M(Q_{1})$, then $\beta(\omega)\leq 1$.
\end{theorem}
\begin{proof}
Put $\alpha=\beta(\omega)$. When $M(P_{\omega})\geq M(Q_{1})$, according to Theorem \ref{t3.8} one gets that
\begin{equation}\label{5.1a}
M(Q_{\alpha})=M(P_{\omega})\geq M(Q_{1}).
\end{equation}
From Proposition \ref{p3.2}, one has that
\begin{equation}\label{5.1b}
\beta(Q_{1})=1.
\end{equation}
From Proposition \ref{p3.3}, one has that
\begin{equation}\label{5.1c}
\alpha\mapsto  M(Q_{\alpha})=\lvert\lvert Q_{\alpha}\lvert\lvert_{L^{2}}^{2} \quad  \text{is strictly increasing on}\;\;[1,\infty).
\end{equation}
Then by (\ref{5.1a}), (\ref{5.1b}) and (\ref{5.1c}), it follows that
\begin{equation}\label{5.1d}
\beta(\omega)=\alpha\geq 1.
\end{equation}
On the other hand, when $\beta(\omega)=\alpha\geq 1$, by (\ref{5.1b}) and (\ref{5.1c}), it follows that
\begin{equation}\label{5.1e}
M(Q_{\alpha})=M(P_{\omega})\geq M(Q_{1}).
\end{equation}
Thus one gets that
\begin{equation}\label{5.1f}
M(P_{\omega})\geq M(Q_{1})\;\;\;\;\text{if and only if}\;\;\;\;\beta(\omega)\geq 1.
\end{equation}
It follows that Theorem \ref{t5.1} is true.
\end{proof}

\begin{theorem}\label{t5.2}
Let $\frac{4}{3\sqrt{3}}M(Q_{1})\leq m< M(Q_{1})$, where $Q_{1}$ as in Proposition \ref{p3.2}. Suppose that both $\xi$ is a positive minimizer of $E_{min}^{V}(m)$ and $\xi$ is a ground state. Then $\xi$ is unique up to translations.
\end{theorem}
\begin{proof}
Since
\begin{equation}\label{5.2a}
\frac{4}{3\sqrt{3}}M(Q_{1})\leq m< M(Q_{1}),
\end{equation}
$E_{min}^{V}(m)$ possesses a positive minimizer $\xi$ by Proposition \ref{p3.6}. From the supposition, $\xi$ is a ground state. Then $\xi$ satisfies the following Euler-Lagrange equation
\begin{equation}\label{5.2b}
-\Delta \xi+\xi^{5}-\xi^{3}+\omega' \xi=0
\end{equation}
with the Lagrange multiplier $\omega'\in (0,\frac{3}{16})$. By Proposition \ref{p2.1},
\begin{equation}\label{5.2c}
\xi=P_{\omega'}\;\;\;\text{up to transalations}.
\end{equation}
In addition, $P_{\omega'}$ satisfies
\begin{equation}\label{5.2d}
M(P_{\omega'})=M(\xi)=m,\;\;\;V(P_{\omega'})=V(\xi)=0\;\;\;\;E(P_{\omega'})=E(\xi)=E_{min}^{V}(m).
\end{equation}
Now put
\begin{equation}\label{5.2e}
\alpha=\beta(\xi)=\beta(P_{\omega'})=\beta(\omega').
\end{equation}
By (\ref{5.2a}) and Theorem \ref{t5.1}, one implies that
\begin{equation}\label{5.2f}
0<\alpha<1.
\end{equation}
From Theorem \ref{t3.9}, there exists a unique $Q_{\alpha}(x)>0$ such that
\begin{equation}\label{5.2g}
-\Delta Q_{\alpha}+Q_{\alpha}^{5}-Q_{\alpha}^{3}+\omega Q_{\alpha}=0
\end{equation}
with
\begin{equation}\label{5.2h}
\omega=\frac{1+\alpha}{3}\frac{\lvert\lvert \nabla Q_{\alpha}\lvert\lvert_{L^{2}}^{2}}{\lvert \lvert Q_{\alpha}\lvert \lvert_{L^{2}}^{2}}\;\;\;\text{and}\;\;\;P_{\omega}=Q_{\alpha}\;\;\;\text{up to translations}.
\end{equation}
In addition,
\begin{equation}\label{5.2i}
\beta(Q_{\alpha})=\alpha=\beta(P_{\omega})=\beta(\omega).
\end{equation}
Since $\alpha<1$, by Theorem \ref{t3.10}
\begin{equation}\label{5.2j}
\alpha\mapsto \frac{\lvert\lvert \nabla Q_{\alpha}\lvert\lvert_{L^{2}}^{2}}{\lvert \lvert Q_{\alpha}\lvert \lvert_{L^{2}}^{2}} \quad  \text{is strictly increasing}.
\end{equation}
It follows that
\begin{equation}\label{5.2k}
\text{the mapping}\;\;\;\alpha\mapsto \frac{1+\alpha}{3}\frac{\lvert\lvert \nabla Q_{\alpha}\lvert\lvert_{L^{2}}^{2}}{\lvert \lvert Q_{\alpha}\lvert \lvert_{L^{2}}^{2}}:=f(\alpha) \quad  \text{is injective on}\;\;(0,1).
\end{equation}
Therefore by (\ref{5.2h}) and (\ref{5.2k}),
\begin{equation}\label{5.2l}
\omega=f(\alpha)\;\;\;\text{is an injective function from}\;\;(0,1)\;\;\;\text{to}\;\;\;(0,\frac{3}{16}).
\end{equation}
By (\ref{5.2e}) and (\ref{5.2i}),
\begin{equation}\label{5.2m}
\beta(\omega')=\beta(\omega)=\alpha.
\end{equation}
(\ref{5.2l}) and (\ref{5.2m}) imply that
\begin{equation}\label{5.2n}
\omega'=\omega.
\end{equation}
Therefore $\xi=P_{\omega'}=P_{\omega}$ (up to translations) is unique.
\end{proof}

\begin{theorem}\label{t5.3}
Let $\frac{4}{3\sqrt{3}}M(Q_{1})\leq m< M(Q_{1})$, where $Q_{1}$ as in Proposition \ref{p3.2}. Suppose that both $\xi$ is a positive minimizer of $E_{min}^{V}(m)$ and $\xi$ is a rescaled soliton. Then $\xi$ is unique up to translations.
\end{theorem}
\begin{proof}
Since $\frac{4}{3\sqrt{3}}M(Q_{1})\leq m<M(Q_{1})$, according to Proposition \ref{p3.6}, $E_{min}^{V}(m)$ possesses a positive minimizer $\xi$. From the supposition, $\xi$ is a rescaled soliton. Now suppose that $\xi_{1}$ and $\xi_{2}$ are two positive minimizers of $E_{min}^{V}(m)$ and both $\xi_{1}$ and $\xi_{2}$ are resacled solitons. Then
\begin{equation}\label{5.3a}
\beta (\xi_{1})=\frac{1}{3}=\beta(\xi_{2}),
\end{equation}
\begin{equation}\label{5.3b}
\int \lvert \xi_{1}\lvert^{2}dx=m=\int \lvert \xi_{2}\lvert^{2}dx,
\end{equation}
\begin{equation}\label{5.3c}
E(\xi_{1})=\frac{1}{9}\int \lvert \nabla \xi_{1}\lvert^{2}dx=E_{min}^{V}(m)=E(\xi_{2})=\frac{1}{9}\int \lvert \nabla \xi_{2}\lvert^{2}dx.
\end{equation}
\begin{equation}\label{5.3d}
V(\xi_{1})=0=V(\xi_{2}).
\end{equation}
Since $\xi_{1}$ and $\xi_{2}$ are resacled solitons, then there exist the ground state $P_{\omega_{1}}$ and the ground state $P_{\omega_{2}}$ such that
\begin{equation}\label{5.3e}
\xi_{1}=\sqrt{\frac{1+\beta(\omega_{1})}{4\beta(\omega_{1})}}P_{\omega_{1}}(\frac{3[1+\beta(\omega_{1})]}{4\sqrt{3\beta(\omega_{1})}}x),  \quad \xi_{2}=\sqrt{\frac{1+\beta(\omega_{2})}{4\beta(\omega_{2})}}P_{\omega_{2}}(\frac{3[1+\beta(\omega_{2})]}{4\sqrt{3\beta(\omega_{2})}}x).
\end{equation}
Denote
\begin{equation}\label{5.3f}
\beta(\omega_{1})=\beta(P_{\omega_{1}})=\alpha,\quad \beta(\omega_{2})=\beta(P_{\omega_{2}})=\nu.
\end{equation}
Then $0<\alpha<\infty,\; 0<\nu<\infty$. According to Proposition \ref{p3.1}, one has that
\begin{equation}\label{5.3g}
\beta(Q_{\alpha})=\alpha,\quad \beta(Q_{\nu})=\nu.
\end{equation}
By (\ref{5.3f}) and (\ref{5.3g}),
\begin{equation}\label{5.3h}
\beta(Q_{\alpha})=\alpha=\beta(P_{\omega_{1}}),\quad \beta(Q_{\nu})=\nu=\beta(P_{\omega_{2}}).
\end{equation}
In terms of Theorem \ref{t3.9}, from $\beta(Q_{\alpha})=\alpha$ it follows that up to translations
\begin{equation}\label{5.3i}
Q_{\alpha}=P_{\omega}\;\;\; \text{with}\;\;\; \omega=\frac{1+\alpha}{3}\frac{\lvert\lvert \nabla Q_{\alpha}\lvert\lvert_{L^{2}}^{2}}{\lvert \lvert Q_{\alpha}\lvert \lvert_{L^{2}}^{2}}.
\end{equation}
By (\ref{5.3h}) and (\ref{5.3i}), one has that
\begin{equation}\label{5.3ia}
\beta(P_{\omega})=\beta(P_{\omega_{1}}),\;\;\; \omega, \; \omega_{1}\in (0,\frac{3}{16}).
\end{equation}
In the following, we divide two situations to proceed.
\par The first situation is that
\begin{equation}\label{5.3ib}
\beta(P_{\omega})=\beta(P_{\omega_{1}})\geq 1.
\end{equation}
Since $\beta(\omega^{\ast})=1$, Theorem \ref{t4.5} implies that
\begin{equation}\label{5.3ic}
\beta(\omega)\;\;\;\text{establishes a one-to-one mapping from}\;\;\;[\omega^{\ast},\;\frac{3}{16})\;\;\;\text{to}\;\;\;[1,\; \infty).
\end{equation}
Therefore by (\ref{5.3ia}), it follows that
\begin{equation}\label{5.3id}
P_{\omega}=P_{\omega_{1}} \;\;\;\text{up to translations.}
\end{equation}
\par The second situation is that
\begin{equation}\label{5.3ie}
\beta(P_{\omega})=\beta(P_{\omega_{1}})=\alpha<1.
\end{equation}
According to (\ref{5.3i}) and Theorem \ref{t3.8},
\begin{equation}\label{5.3if}
\omega=\frac{1+\alpha}{3}\frac{\lvert\lvert \nabla Q_{\alpha}\lvert\lvert_{L^{2}}^{2}}{\lvert \lvert Q_{\alpha}\lvert \lvert_{L^{2}}^{2}}=\frac{16\alpha}{9(1+\alpha)}\frac{\lvert\lvert \nabla R_{\alpha}\lvert\lvert_{L^{2}}^{2}}{\lvert \lvert R_{\alpha}\lvert \lvert_{L^{2}}^{2}}.
\end{equation}
In terms of Theorem \ref{t3.11},
\begin{equation}\label{5.3ig}
\alpha\mapsto \lvert\lvert \nabla R_{\alpha}\lvert\lvert_{L^{2}}^{2} \quad  \text{is strictly increasing on}\; (0,\infty),
\end{equation}
\begin{equation}\label{5.3ih}
\alpha\mapsto M(R_{\alpha})=\lvert\lvert R_{\alpha}\lvert\lvert_{L^{2}}^{2} \quad  \text{is strictly decreasing on}\;(0,1].
\end{equation}
Since
\begin{equation}\label{5.3ii}
\alpha\mapsto \frac{16\alpha}{9(1+\alpha)} \quad  \text{is strictly increasing on}\;(0,\infty),
\end{equation}
then (\ref{5.3ig}), (\ref{5.3ih}) and (\ref{5.3ii}) imply that
\begin{equation}\label{5.3ij}
\omega=\frac{16\alpha}{9(1+\alpha)}\frac{\lvert\lvert \nabla R_{\alpha}\lvert\lvert_{L^{2}}^{2}}{\lvert \lvert R_{\alpha}\lvert \lvert_{L^{2}}^{2}}:=f(\alpha)  \quad  \text{is strictly increasing on}\;(0,1].
\end{equation}
By (\ref{5.3ie}), it follows that
\begin{equation}\label{5.3ik}
\omega=\omega_{1}, \;\;\;\text{that is}\;\;\;\;P_{\omega}=P_{\omega_{1}}\;\;\;\text{up to translations}.
\end{equation}
Therefore one always has that
\begin{equation}\label{5.3il}
Q_{\alpha}=P_{\omega_{1}}\;\;\;\text{up to translations}.
\end{equation}
By using the same argument, one also gets
\begin{equation}\label{5.3im}
Q_{\nu}=P_{\omega_{2}}\;\;\;\text{up to translations}.
\end{equation}
\par Now define
\begin{equation}\label{5.3j}
R_{\alpha}=\sqrt{\frac{1+\alpha}{4\alpha}}P_{\omega_{1}}(\frac{3[1+\alpha]}{4\sqrt{3\alpha}}x),  \quad R_{\nu}=\sqrt{\frac{1+\nu}{4\nu}}P_{\omega_{2}}(\frac{3[1+\nu]}{4\sqrt{3\nu}}x).
\end{equation}
Thus by (\ref{5.3e}), (\ref{5.3f}) and (\ref{5.3j}), one gets that up to translations
\begin{equation}\label{5.3k}
R_{\alpha}=\xi_{1},\quad R_{\nu}=\xi_{2}.
\end{equation}
By (\ref{5.3c}), one implies that
\begin{equation}\label{5.3l}
\int \lvert \nabla\xi_{1}\lvert^{2}dx=\int \lvert \nabla\xi_{2}\lvert^{2}dx.
\end{equation}
From (\ref{5.3k}) and (\ref{5.3l}), it follows that
\begin{equation}\label{5.3m}
\int \lvert \nabla R_{\alpha}\lvert^{2}dx=\int \lvert \nabla R_{\nu}\lvert^{2}dx.
\end{equation}
By Theorem \ref{t3.11}, the monotonicity of $\lvert \lvert \nabla R_{\alpha}\lvert \lvert_{L^{2}}^{2}$ implies that $\alpha=\nu$. By (\ref{5.3k}), it follows that
$\xi_{1}=\xi_{2}$ up to translations. It shows that $E_{min}^{V}(m)$ with $\frac{4}{3\sqrt{3}}M(Q_{1})\leq m<M(Q_{1})$ possesses a unique positive minimizer up to translations.
\end{proof}

\begin{theorem}\label{t5.4}
Let $0<\alpha<1$ and $R_{\alpha}$ as in Theorem \ref{t3.8}. Put
\begin{equation*}
m=M(R_ {\alpha})=\lvert\lvert R_{\alpha}\lvert\lvert_{L^{2}}^{2}.
\end{equation*}
Suppose that $\xi$ is a positive minimizer of $E_{min}^{V}(m)$ and $\xi$ is a rescaled soliton. Then $\xi=R_{\alpha}$ up to translations.
\end{theorem}
\begin{proof}
By Theorem \ref{t3.8}, $R_{\alpha}$ is a rescaled soliton. From Proposition \ref{p3.4}, one has that
\begin{equation}\label{5.4a}
m=M(R_{\alpha})\geq \frac{4}{3\sqrt{3}}M(Q_{1}).%\leq M(Q_{1}).
\end{equation}
According to Proposition \ref{p3.6}, $E_{min}^{V}(m)$ possesses a positive minimizer $\xi$. By the supposition, $\xi$ is a rescaled soliton. Thus $\beta(\xi)=\frac{1}{3}$. It follows that
\begin{equation}\label{5.4b}
E(\xi)=\frac{1}{9}\int \lvert \nabla \xi\lvert^{2}dx.
\end{equation}
By Theorem \ref{t3.8}, $V(R_{\alpha})=0$. Since $M(R_{\alpha})=m$, it follows that
\begin{equation}\label{5.4c}
E(R_{\alpha})\geq E_{min}^{V}(m)=E(\xi)=\frac{1}{9}\int \lvert \nabla \xi\lvert^{2}dx.
\end{equation}
According to (\ref{1.12}) and Theorem \ref{t3.8},
\begin{equation}\label{5.4d}
E(R_{\alpha})=\frac{1}{9}\int \lvert \nabla R_{\alpha}\lvert^{2}dx.
\end{equation}
From (\ref{5.4c}) and (\ref{5.4d}), it follows that
\begin{equation}\label{5.4e}
\int \lvert \nabla R_{\alpha}\lvert^{2}dx \geq \int \lvert \nabla \xi\lvert^{2}dx.
\end{equation}
By Theorem \ref{t3.8}, $R_{\alpha}$ is a positive minimizer of the variational problem (\ref{1.9}) and
\begin{equation}\label{5.4f}
C_{\alpha}^{-1}=F_{\alpha}(R_{\alpha})=\frac{9}{16}\big(\frac{1}{3}\big)^{\frac{\alpha}{2(1+\alpha)}}\big(\int \lvert R_{\alpha}\lvert^{2}dx\big)^{\frac{1}{2}}\big(\int \lvert \nabla R_{\alpha}\lvert^{2}dx\big)^{\frac{1-\alpha}{2(1+\alpha)}}.
\end{equation}
It follows that
\begin{equation}\label{5.4g}
F_{\alpha}(\xi)\geq F_{\alpha}(R_{\alpha}).
\end{equation}
Since $\beta(\xi)=\frac{1}{3}$ and $V(\xi)=0$, it follows that
\begin{equation}\label{5.4h}
F_{\alpha}(\xi)=\frac{9}{16}\big(\frac{1}{3}\big)^{\frac{\alpha}{2(1+\alpha)}}\big(\int \lvert \xi\lvert^{2}dx\big)^{\frac{1}{2}}\big(\int \lvert \nabla \xi\lvert^{2}dx\big)^{\frac{1-\alpha}{2(1+\alpha)}}.
\end{equation}
By (\ref{5.4f}), (\ref{5.4h}), (\ref{5.4g}), $M(\xi)=m=M(R_{\alpha})$ and $0<\alpha<1$,
one gets that
\begin{equation}\label{5.4i}
\int \lvert \nabla \xi\lvert^{2}dx \geq \int \lvert \nabla R_{\alpha}\lvert^{2}dx.
\end{equation}
From (\ref{5.4e}) and (\ref{5.4i}), it follows that
\begin{equation}\label{5.4j}
\int \lvert \nabla R_{\alpha}\lvert^{2}dx =\int \lvert \nabla \xi\lvert^{2}dx.
\end{equation}
From (\ref{5.4j}) and $M(\xi)=m=M(R_{\alpha})$, one has that
\begin{equation}\label{5.4k}
F_{\alpha}(\xi)=F_{\alpha}(R_{\alpha})=C_{\alpha}^{-1}.
\end{equation}
Therefore $\xi$ is a positive minimizer of $C_{\alpha}^{-1}$. From (\ref{5.4b}), (\ref{5.4d}) and (\ref{5.4j}), it follows that
\begin{equation}\label{5.4l}
E(R_{\alpha})=E(\xi)=E_{min}^{V}(m).
\end{equation}
Therefore $R_{\alpha}$ is also a positive minimizer of $E_{min}^{V}(m)$ with $m=M(R_{\alpha})$. In terms of Theorem \ref{t5.3}, one gets that $\xi=R_{\alpha}$ up to translations.
\end{proof}

\begin{theorem}\label{t5.5}
There exists $\omega_{\ast}\in (0,\frac{3}{16})$ such that $\beta(\omega_{\ast})=\frac{1}{3}$. Denote $m_{0}=M(P_{\omega_{\ast}})$. Then $E_{min}^{V}(m_{0})$ possesses a unique positive minimizer $P_{\omega_{\ast}}=R_{\omega_{\ast}}$ up to translations.
\end{theorem}
\begin{proof}
Since $\beta(\omega)$ is continuous from $(0,\frac{3}{16})$ to $(0,\infty)$, there exists $\omega_{\ast}\in (0,\frac{3}{16})$ such that $\beta(\omega_{\ast})=\frac{1}{3}$. According to Proposition \ref{p2.3}, $\beta(\omega_{\ast})=\frac{1}{3}$ implies that
\begin{equation}\label{5.5a}
R_{\omega_{\ast}}=P_{\omega_{\ast}}\;\;\;\text{up to translations}.
\end{equation}
Now suppose that the positive minimizer $\xi$ of $E_{min}^{V}(m_{0})$ is a rescaled soliton. By Theorem \ref{t5.4}, $R_{\omega_{\ast}}=R_{\alpha}$ (up to translations) with $\alpha=\frac{1}{3}$ is a unique positive minimizer of $E_{min}^{V}(m_{0})$. From (\ref{5.5a}), $P_{\omega_{\ast}}$ is also a positive minimizer of $E_{min}^{V}(m_{0})$. Since $P_{\omega_{\ast}}$ is a ground state, Theorem \ref{t5.2} implies that $P_{\omega_{\ast}}$ is a unique positive minimizer of $E_{min}^{V}(m_{0})$.
\end{proof}

\begin{theorem}\label{t5.6}
Let $\alpha\in [\frac{1}{3},\; 1]$. Then
\begin{equation*}
\omega=\frac{1+\alpha}{3}\frac{\lvert \lvert \nabla Q_{\alpha}\lvert\lvert_{L^{2}}^{2}}{\lvert \lvert Q_{\alpha}\lvert\lvert_{L^{2}}^{2}}:=f(\alpha)
\end{equation*}
is a one-to-one mapping from $[\frac{1}{3},\; 1]$ to $[\omega_{\ast},\; \omega^{\ast}]$. In addition, $\beta(\omega)$ is strictly increasing on $[\omega_{\ast},\; \omega^{\ast}]$.
\end{theorem}
\begin{proof}
In terms of Theorem \ref{t3.10},
\begin{equation}\label{5.6a}
\omega=f(\alpha)=\frac{1+\alpha}{3}\frac{\lvert \lvert \nabla Q_{\alpha}\lvert\lvert_{L^{2}}^{2}}{\lvert \lvert Q_{\alpha}\lvert\lvert_{L^{2}}^{2}}
\end{equation}
is strictly increasing on $[\frac{1}{3},\;1]$. Thus $\omega=f(\alpha)$ is an injective mapping on $[\frac{1}{3}, 1]$. According to Theorem \ref{t3.8}, one has that
\begin{equation}\label{5.6b}
\alpha=\beta(Q_{\alpha})=\beta(P_{\omega})=\beta(\omega).
\end{equation}
By $\beta(\omega_{\ast})=\frac{1}{3}$ and $\beta(\omega^{\ast})=1$, one has that $\omega=f(\alpha)$ is a surjective mapping from $[\frac{1}{3},\; 1]$ to $[\omega_{\ast},\; \omega^{\ast}]$. Therefore $\omega=f(\alpha)$ is a one-to-one mapping from $[\frac{1}{3},\; 1]$ to $[\omega_{\ast},\; \omega^{\ast}]$. It follows that
\begin{equation}\label{5.6c}
\beta(\omega)\;\;\;\text{is strictly increasing on}\;\;\; [\omega_{\ast},\; \omega^{\ast}].
\end{equation}
\end{proof}

\begin{theorem}\label{t5.7}
Let $\omega\in [\omega_{\ast},\; \omega^{\ast}]$. Then $M(R_{\omega})\leq m_{0}=M(P_{\omega_{\ast}})$.
\end{theorem}
\begin{proof}
In terms of Theorem \ref{t3.11}, one has that
\begin{equation}\label{5.7a}
\alpha\mapsto M(R_{\alpha})=\lvert\lvert R_{\alpha}\lvert\lvert_{L^{2}}^{2} \quad  \text{is strictly decreasing on}\;[\frac{1}{3},\;1].
\end{equation}
By Theorem \ref{t3.8},
\begin{equation}\label{5.7b}
R_{\alpha}=R_{\omega}\;\;\;\text{up to translations}.
\end{equation}
Thus
\begin{equation}\label{5.7c}
m_{0}=M(P_{\omega_{\ast}})=M(R_{\omega_{\ast}}).
\end{equation}
By Theorem \ref{t5.6}, (\ref{5.7a}), (\ref{5.7b}) and (\ref{5.7c}),
\begin{equation}\label{5.7d}
\omega\mapsto M(R_{\omega})=\lvert\lvert R_{\omega}\lvert\lvert_{L^{2}}^{2} \quad  \text{is strictly decreasing on}\;[\omega_{\ast},\;\omega^{\ast}].
\end{equation}
Therefore, for $\omega\in [\omega_{\ast},\;\omega^{\ast}]$,
\begin{equation}\label{5.7e}
M(R_{\omega})\leq m_{0}.
\end{equation}
\end{proof}

\begin{theorem}\label{t5.8}
Let $m_{0}<m< M(Q_{1})$. Then $E_{min}^{V}(m)$ possesses a unique positive minimizer which is a ground state.
\end{theorem}
\begin{proof}
Let $\xi$ be a positive minimizer of $E_{min}^{V}(m)$. Then
\begin{equation}\label{5.8a}
M(\xi)=m>m_{0}.
\end{equation}
According to Theorem \ref{t5.7}, $\xi$ is not a rescaled soliton. By Proposition \ref{p3.7}, $\xi$ is a ground state. In terms of Theorem \ref{t5.2}, $\xi$ is unique.
\end{proof}

\section{Monotonicity of frequency}\label{sec6}

\begin{theorem}\label{t6.1}
$\beta(\omega)$ is strictly increasing on $(0,\;\omega_{\ast}]$.
\end{theorem}
\begin{proof}
According to Proposition \ref{p2.2},
\begin{equation}\label{6.1a}
\frac{d}{d\omega}M(P_{\omega})<\frac{3\beta(\omega)-1}{2\omega}M(P_{\omega}).
\end{equation}
Let $\alpha=\beta(\omega)$. From $0<\alpha \leq\frac{1}{3}$ and $\beta(\omega_{\ast})=\frac{1}{3}$,  it follows that
\begin{equation}\label{6.1b}
\frac{d}{d\omega}M(P_{\omega})\leq 0,\;\;\; \omega\in (0,\omega_{\ast}].
\end{equation}
From Proposition \ref{p4.4},it follows that
\begin{equation}\label{6.1b1}
\frac{d}{d\omega}M(P_{\omega})<0,\;\;\;\text{almost every}\;\;\; \omega\in (0,\omega_{\ast}].
\end{equation}
According to Proposition \ref{p3.3},
\begin{equation}\label{6.1c}
\alpha \mapsto M(Q_{\alpha})=\lvert\lvert Q_{\alpha}\lvert\lvert_{L^{2}}^{2} \quad \text{is strictly decreasing on}\;  (0,\frac{1}{3}].
\end{equation}
It follows that
\begin{equation}\label{6.1d}
\frac{dM}{d\alpha}<0,\;\;\;\text{almost every}\;\;\; \alpha\in (0,\;\frac{1}{3}].
\end{equation}
Since
\begin{equation}\label{6.1e}
\frac{dM}{d\omega}=\frac{dM}{d\alpha}\cdot \frac{d\beta}{d\omega},
\end{equation}
by (\ref{6.1b1}) and (\ref{6.1d}) it follows that
\begin{equation}\label{6.1f}
\frac{d\beta}{d\omega}>0,\;\;\;\text{almost every}\;\;\; \omega\in (0,\omega_{\ast}].
\end{equation}
Thus (\ref{6.1f}) implies that
\begin{equation}\label{6.1g}
\beta(\omega)\;\;\;\text{is strictly increasing on}\;\;\; (0,\;\omega_{\ast}].
\end{equation}
\end{proof}

Now we complete the proof of Theorem \ref{t1.1}.
\begin{proof}
Theorem \ref{t6.1}, Theorem \ref{t5.6} and Theorem \ref{t4.5} show that Theorem \ref{t1.1} is true.
\end{proof}
%By Theorem \ref{t3.8},
%\begin{equation}\label{6.1a}
%Q_{\alpha}=P_{\omega}.
%\end{equation}

\section{Monotonicity of mass}\label{sec7}

\begin{theorem}\label{t7.1}
$M(P_{\omega})$ is strictly increasing on $\omega\in [\omega_{\ast},\;\omega^{\ast})$.
\end{theorem}
\begin{proof}
In terms of Theorem \ref{t3.8},
\begin{equation}\label{7.1a}
P_{\omega^{\ast}}=Q_{1}\;\;\;\text{up to translations}.
\end{equation}
It follows that
\begin{equation}\label{7.1b}
M(P_{\omega^{\ast}})=M(Q_{1}).
\end{equation}
From Theorem \ref{t5.5},
\begin{equation}\label{7.1c}
M(P_{\omega_{\ast}})=m_{0}.
\end{equation}
By Theorem  \ref{t5.8}, (\ref{7.1b}) and (\ref{7.1c}),  $M(P_{\omega})$ establishes a one-to-one mapping from $(M(P_{\omega_{\ast}}),\; M(P_{\omega^{\ast}}))$ to $(\omega_{\ast},\;\omega^{\ast})$. It follows that
\begin{equation}\label{7.1d}
M(P_{\omega})=M(\omega)\;\;\; \text{is strictly increasing on} \;\;\;[\omega_{\ast},\;\omega^{\ast}].
\end{equation}
\end{proof}

\begin{theorem}\label{t7.2}
$M(P_{\omega})$ is strictly decreasing on $\omega\in (0,\; \omega_{\ast})$.
\end{theorem}
\begin{proof}
(\ref{6.1b}) and Proposition \ref{p4.4} imply that
\begin{equation}\label{6.1g}
M(P_{\omega})\;\;\;\text{is strictly decreasing on}\;\;\; (0,\;\omega_{\ast}).
\end{equation}
\end{proof}

Now we complete the proof of Theorem \ref{t1.2}.
\begin{proof}
Theorem \ref{t7.2}, Theorem \ref{t7.1} and Theorem \ref{t4.3} show that Theorem \ref{t1.2} is true.
\end{proof}

\section{Uniqueness of energy minimizer}\label{sec8}

\begin{theorem}\label{t8.1}
Let $\frac{4}{3\sqrt{3}}M(Q_{1})\leq m<m_{0}$. Then $E_{min}^{V}(m)$ possesses a unique positive minimizer which is a rescaled soliton.
\end{theorem}
\begin{proof}
Let $\xi$ be a positive minimizer of $E_{min}^{V}(m)$. Then $\xi$ obeys
\begin{equation}\label{8.1a}
\frac{4}{3\sqrt{3}}M(Q_{1})\leq M(\xi)=m<m_{0}=M(P_{\omega_{\ast}}),\;\;\; E(\xi)=E_{min}^{V}(m),\;\;\; V(\xi)=0.
\end{equation}
By Proposition \ref{p3.7} and Theorem \ref{t7.1}, one yields that $\xi$ must be a rescaled soliton. According to Theorem \ref{t5.3}, $\xi$ is unique up to translations.
\end{proof}

Now we complete the proof of Theorem \ref{t1.3}.
\begin{proof}

Theorem \ref{t8.1}, Theorem \ref{t5.5}, Theorem \ref{t5.8} and Theorem \ref{t4.2} show that Theorem \ref{t1.3} is true.
\end{proof}

\section{Sharp stability of solitons}\label{sec7}

According to to Killip, Oh, Pocovnicu and Visan \cite{KOPV2017} as well as Zhang \cite{Z2006}, the following well-posedness is true.

\begin{proposition}\label{p9.1}
(\cite{KOPV2017}) For arbitrary $\varphi_0\in H^1(\mathbb{R}^{3})$, (\ref{1.1}) possesses a unique global solution $\varphi\in \mathcal{C}(\mathbb{R}; H^1(\mathbb{R}^{3}))$ such that $\varphi(0,x)=\varphi_0$. In addition, the solution holds the conservation of mass, energy and momentum, where mass is given by (\ref{1.11}), energy is given by (\ref{1.12}) and momentum is given by $Q(u):=\int2 \mathrm{Im}(\bar{u}\nabla u)dx$ for $u\in H^1(\mathbb{R}^{3})$.
\end{proposition}

Let $\omega\in(0,\frac{3}{16})$ and $P_{\omega}(x)$ be the unique positive solution of (\ref{1.2}). Then one has that
\begin{equation}\label{9.1a}
-\Delta P_{\omega}+\omega P_{\omega}-P_{\omega}^3+P_{\omega}^5=0, \quad P_{\omega}\in H^1(\mathbb{R}^3).
\end{equation}
The linearized operator of (\ref{9.1a}) around $P_\omega$ is that
\begin{equation}\label{9.1b}
H_{\omega}=-\Delta+\omega-3P_{\omega}^2+5P_{\omega}^4.
\end{equation}
It is clear that
\begin{equation}\label{9.1c}
H_{\omega}=E^{\prime\prime}(P_{\omega})+\frac{1}{2}\omega M^{\prime\prime}(P_{\omega}).
\end{equation}

According to Killip, Oh, Pocovnicu and Visan \cite{KOPV2017}, the following Propositions are true.

\begin{proposition}\label{p9.2}
(\cite{KOPV2017}) The operator $H_{\omega}$ has one negative simple eigenvalue and has its kernel spanned by $iP_\omega$. Moreover the positive spectrum of $H_\omega$ is bounded away from zero.
\end{proposition}

\begin{proposition}\label{p9.3}
(\cite{KOPV2017}) Let $\omega\in(0,\frac{3}{16})$ and $P_\omega$ be the ground state of (\ref{1.2}). Define the scalar function
\begin{equation*}
d(\omega)=E(P_{\omega})+\frac{\omega}{2}M(P_{\omega}),
\end{equation*}
Then one has that
\begin{equation*}
d''(\omega)=\frac{1}{2}\frac{d}{d\omega}M(P_{\omega}).
\end{equation*}
\end{proposition}

By Proposition \ref{p9.2}, $H_\omega$ with $T^{\prime}(0)=i$ satisfies Assumption 3 in Grillakis-Shatah-Strauss \cite{GSS1987} for $\omega\in(0,\frac{3}{16})$. With $J=-i$, $X=H^1(\mathbb{R}^3)$
and  $E$ defined as (\ref{1.12}), by Proposition \ref{p9.1} and Proposition \ref{p2.1}, (\ref{1.1}) satisfies
Assumption 1 and Assumption 2 in Grillakis-Shatah-Strauss \cite{GSS1987} for $\omega\in(0,\frac{3}{16})$.
Thus in terms of \cite{GSS1987}, the following proposition is true.

\begin{proposition}\label{p9.4}
(\cite{KOPV2017}) Let $\omega\in(0,\frac{3}{16})$ and $P_\omega$ be the ground state of (\ref{1.2}). Then the soliton $P_{\omega}e^{i\omega t}$ of (\ref{1.1}) is orbitally stable if and only if the function $d(\cdot)$ is convex in a neighborhood of $\omega$.
\end{proposition}

\begin{theorem}\label{t9.5}
For $\omega_{\ast}$ in Theorem \ref{t5.5}, the soliton $P_{\omega_{\ast}}e^{i\omega_{\ast}t}$ of (\ref{1.1}) is unstable.
\end{theorem}
\begin{proof}
In terms of Theorem \ref{t1.2}, one has that
\begin{equation}\label{9.5a}
\frac{d}{d\omega}M(P_\omega)\mid_{\omega=\omega_\ast}=0.
\end{equation}
From Theorem \ref{t1.2},
\begin{equation}\label{9.5b}
\frac{d}{d\omega}M(P_\omega)\leq 0 \;\;\;\text{for}\;\;\;\;\omega\in (0,\omega_{\ast}).
\end{equation}
From Theorem \ref{t1.2},
\begin{equation}\label{9.5c}
\frac{d}{d\omega}M(P_\omega)\geq 0 \;\;\;\text{for}\;\;\;\;\omega\in (\omega_{\ast},\frac{3}{16}).
\end{equation}
By (\ref{9.5a}) and Proposition \ref{p4.4}, $\omega_\ast$ is the isolated zero point of $\frac{d}{d\omega}M(P_\omega)$. Thus by (\ref{9.5b}) and (\ref{9.5c}), there exists $\delta>0$ such that
\begin{equation}\label{9.5d}
\frac{d}{d\omega}M(P_\omega)< 0 \;\;\; \text{for}\;\;\;  \omega\in(\omega_\ast-\delta,\omega_\ast);
\end{equation}
\begin{equation}\label{9.5e}
\frac{d}{d\omega}M(P_\omega)> 0\;\;\; \text{for}\;\;\; \omega\in(\omega_\ast,\omega_\ast+\delta).
\end{equation}
By Proposition \ref{p9.3}, (\ref{9.5a}), (\ref{9.5d}) and (\ref{9.5e}), one has that $d''(\omega_{\ast})=0$,
\begin{equation}\label{9.5f}
d''(\omega)<0 \;\;\;\text{for}\;\;\;\omega\in(\omega_\ast-\delta,\omega_\ast);
\end{equation}
\begin{equation}\label{9.5g}
d''(\omega)>0 \;\;\;\text{for}\;\;\;\omega\in(\omega_\ast,\omega_\ast+\delta).
\end{equation}
Therefore $d(\cdot)$ is not convex in a neighborhood of $\omega_{\ast}$. By Proposition \ref{p9.4}, the soliton $P_{\omega_{\ast}}e^{i\omega_{\ast}t}$ of (\ref{1.1}) is unstable.
\end{proof}

\begin{theorem}\label{t9.6}
Let $\omega\in(0,\frac{3}{16})$ and $P_\omega$ be the ground state of (\ref{1.2}).
Then for $\omega\in(0,\omega_*)$, the soliton $P_{\omega}e^{i\omega t}$ of (\ref{1.1}) is unstable.
\end{theorem}

\begin{proof}
Since $\beta(\omega_{\ast})=\frac{1}{3}$, from Proposition \ref{p2.2},
\begin{equation}\label{9.6a}
\frac{d}{d\omega}M(P_\omega)<0 \;\;\;\text{for}\;\;\;\;\omega\in (0,\omega_{\ast}).
\end{equation}
By Proposition \ref{p9.3}, it follows that
\begin{equation}\label{9.6b}
d''(\omega)< 0 \;\;\;\text{for}\;\;\;\omega\in(0,\omega_\ast).
\end{equation}
Therefore by (\ref{9.6b}), $d(\cdot)$ is not convex on $(0,\omega_{\ast})$. By Proposition \ref{p9.4}, the soliton $P_{\omega}e^{i\omega t}$ of (\ref{1.1}) is unstable for $\omega\in (0,\omega_{\ast})$.
\end{proof}

\begin{theorem}\label{t9.7}
Let $\omega\in(0,\frac{3}{16})$ and $P_\omega$ be the ground state of (\ref{1.2}).
Then for $\omega\in(\omega_*,\frac{3}{16})$,  the soliton $P_{\omega}e^{i\omega t}$ of (\ref{1.1}) is orbitally stable.
\end{theorem}
\begin{proof}
From Theorem \ref{t1.2},
\begin{equation}\label{9.7a}
\frac{d}{d\omega}M(P_\omega)\geq 0 \;\;\;\text{for}\;\;\;\;\omega\in (\omega_{\ast},\frac{3}{16}).
\end{equation}
By Proposition \ref{p9.3}, it follows that
\begin{equation}\label{9.7b}
d''(\omega)\geq 0 \;\;\;\text{for}\;\;\;\omega\in(\omega_{\ast},\frac{3}{16}).
\end{equation}
By Proposition \ref{p4.4}, the zero points of $\frac{d}{d\omega}M(P_\omega)$ on $(\omega_{\ast},\frac{3}{16})$ are isolated and countable. From Proposition \ref{p9.3}, it follows that the zero points of $d''(\omega)$ on $(\omega_{\ast},\frac{3}{16})$ are also isolated and countable.
Therefore by (\ref{9.7b}), $d(\cdot)$ is convex on $(\omega_{\ast},\frac{3}{16})$. By Proposition \ref{p9.4}, the soliton $P_{\omega}e^{i\omega t}$ of (\ref{1.1}) is orbitally stable for $\omega\in (\omega_{\ast},\frac{3}{16})$.
\end{proof}

Now we complete the proof of Theorem \ref{t1.4}.
\begin{proof}

Theorem \ref{t9.5}, Theorem \ref{t9.6} and Theorem \ref{t9.7} show that Theorem \ref{t1.4} is true.
\end{proof}

\section{Classification of normalized solutions }\label{sec8}

\begin{theorem}\label{t10.1}
Let $\omega\in(0,\frac{3}{16})$ and $P_\omega$ be the ground state of (\ref{1.2}). Then $M(P_\omega)=\int P^2_\omega dx$ establishes two 1-1   correspondences from $(0,\omega_*]$ to $(\infty, m_{0}]$ and from $(\omega_*,\frac{3}{16})$ to $(m_{0},\infty)$ respectively.
\end{theorem}

\begin{proof}
Since $\beta(\omega_{\ast})=\frac{1}{3}$, from Proposition \ref{p2.2},
\begin{equation}\label{10.1a}
\frac{dM}{d\omega}<0,\;\;\;\omega\in (0,\omega_{\ast}).
\end{equation}
By $M(P_{\omega_{\ast}})=m_{0}$, $M(P_\omega)$ establishes a 1-1 correspondences from $(0,\omega_*]$  to $(\infty, m_{0}]$. According to Theorem \ref{t1.3}, $M(P_\omega)$ establishes a 1-1 correspondences from $[\omega_*,\frac{3}{16})$ to $[m_{0},\infty)$.
\end{proof}

Now we prove Theorem \ref{t1.5}.
\begin{proof}
Since $M(P_{\omega_{\ast}})=m_{0}$, by Theorem \ref{t1.2},
\begin{equation}\label{10.2a}
M(P_{\omega})\geq m_{0},\;\;\; \omega\in (0,\frac{3}{16}).
\end{equation}
When $0\leq m<m_{0}$, (\ref{1.2}) has no positive normalized solutions with the prescribed mass $\int \lvert u\rvert^2dx=m$. The rest of Theorem \ref{t1.5} can be directly gotten by Theorem \ref{t10.1}.
\end{proof}

\vspace{0.2cm}
\textbf{Acknowledgment.}

The author Jian Zhang is grateful to Professor Rowan Killip, Professor Mathieu Lewin and Professor Kenji Nakanish for valuable discussions with them. He is also grateful Professor Yoshio Tsutsumi for his advice and help.

This research is supported by the National Natural Science Foundation of China 12271080.

\textbf{Data Availability Statement}

Data sharing not applicable to this article as no datasets were generated or analysed during the current study.

%ºóÆÚÔö¼ÓһάÇé¿öµÄ½á¹ûºÍһάµÄÏà¹ØÂÛÎÄ£¬²¹³äÍêÕû


%%\bibliography{sn-bibliography}% common bib file

\begin{thebibliography}{10}
%\bibitem{AH1975}Ablowitz, M., Haberman, R.: Nonlinear evolution equations in two and three dimensions. Phys. Rev. Lett. 35, 1185-1188, 1975.
\bibitem{AGTF2001}Abdullaev, F. K., Gammal, A., Tomio, L., Frederico, T.: Stability of trapped Bose-Einstein condensates. \emph{Phys. Rev. A} \textbf{63}, 043604 (2001)
\bibitem{AM2022} Ardila, A. H., Murphy, J.: Threshold solutions for the 3D cubic-quintic NLS. \emph{Commun. Partial Differ. Equ.} \textbf{48}(5), 819-862  (2023)

\bibitem{BZ2022}Bai, M., Zhang, J.: High-speed excited multi-solitons in competitive power nonlinear Schr\"{o}dinger equations. \emph{Z. Angew. Math. Phys.} \textbf{73}, 141 (2022) https://doi.org/10.1007/s00033-022-01774-0

\bibitem{BJS2016}Bartsch, T., Jeanjean, L., Soave, N.: Normalized solutions for a system of coupled cubic Schr\"odinger equations on $\mathbb{R}^{3}$. \emph{J. Math. Pure. Appl.} \textbf{106}: 583-614 (2016)

\bibitem{BMRV2021}Bartsch, T., Molle, R., Rizzi, M., Verzini, G.: Normalized solutions of mass supercritical Schr\"odinger equations with potential. \emph{Commun. Partial Differ. Equ.} \textbf{46}(9), 1729-1756  (2021)
%\bibitem{BS2017}Bartsch, T., Soave, N. (2017). A natural constraint approach to normalized solutions of nonlinear Schr\"odinger equations and systems. J. Funt. Anal. 272: 4998-5037.
%\bibitem{BS2019}Bartsch, T., Soave, N. (2019). Multiple normalized solutions for a competing system of Schr¡§odinger equations. Calc.
%Var. Partial Differential Equations 58, no. 1, Art. 22, 24 pp.
%\bibitem{BZZ2020}Bartsch, T., Zhong, X., Zou, W. M.: Normalized solutions for a system of coupled Schr\"odinger system. Math. Ann. 380(3-4), 1713-1740, 2020.
%doi: 10.1007/S00208-020-02000-W.
\bibitem{BC1981}Berestycki, H., Cazenave, T.: Instabilit$\acute{e}$ des $\acute{e}$tats stationnaies dans les $\acute{e}$quations de Schr\"odinger et de Klein-Gordon non lin$\acute{e}$aires. \emph{C. R. Acad. Sci. Paris S$\acute{e}$r. I Math.} \textbf{293}, 489-492, 1981.
\bibitem{BL1983}Berestycki, H., Lions, P. L.: Nonlinear scalar field equations. \emph{Arch. Rational Math. Anal.} \textbf{82}, 313-375 (1983)
%\bibitem{BLP1981}Berestycki, H., Lions, P. L. (1981). Peletier, L. A., An ODE approach to the existence of positive solutions for semilinear problems in $\mathbb{R}^{N}$. Indiana University Mat. J. 30: 141-157.
%\bibitem{BF2005} Bouard, A. De, Fukuizumi, R. (2005). Stability of standing waves for nonlinear Schr\"odinger equations with inhomogeneous nonlinearities. Ann. Henri Poincar$\acute{e}$. 6: 1157-1177.
\bibitem{B1999}Bourgain, J.: Global wellposedness of defocusing critical nonlinear Schr\"{o}dinger equation in the radial case. \emph{J. Amer. Math. Soc.} \textbf{12}(1), 145-172 (1999)


%\bibitem{BS2003} Buslaev, V. S., Sulem, C. (2003). On asymptotic stability of solitary waves for nonlinear Schr\"odinger equations. Communications in Partial Differential Equations. 20.3: 419-475.
\bibitem{CS2021}Carles, R., Sparber, C.: Orbital stability vs. scattering in the cubic-quintic Schr\"{o}dinger equation. \emph{Rev. Math. Phys.} \textbf{33}, 2150004 (2021)

\bibitem{C2003}Cazenave, T.: Semilinear Schr\"{o}dinger equations, Courant Lecture Notes in Mathematics, 10. New York University, Courant Institute of Mathematical Sciences, New York; American Mathematical Society, Providence, RI. (2003)
\bibitem{CL1982}Cazenave, T., Lions, P. L.: Orbital stability of standing waves for some nonlinear Schr\"odinger equations. \emph{Commun.  Math. Phys.} \textbf{85}(4), 549-561 (1982)
%\bibitem{CW1990}Cazenave, T., Weissler, F. B. (1990).  The Cauchy problem for the critical nonlinear Schr\"odinger equation in $H^{s}$, Nonlinear Anal. 14(10): 807-836.
%\bibitem{CGNT2007}Chang, S.-M.,  Gustafson, S., Nakanishi, K., Tsai, T.-P. (2007/08).  Spectra of linearized operators for NLS solitary waves. SIAM J. Math. Anal. 39(4): 1070-1111.
\bibitem{CKSTT2008} Colliander, J., Keel. M., Staffilani, G., Takaoka, H., Tao. T.: Global well-posedness and scattering for the energy-critical nonlinear Schr\"{o}dinger equation in $\mathbb{R}^3$. \emph{Ann. of Math.} \textbf{167}(3), 767-865 (2008)


\bibitem{CL2011}C$\hat{o}$te, R., Le Coz, S.: High-speed excited multi-solitons in nonlinear Schr\"odinger equations. \emph{J. Math. Pures Appl.} \textbf{96}, 135-166 (2011)

\bibitem{CERS1986}Cowan, S., Enns, H., Rangnekar, S. S. and Sanghera, S. S. B.,  Quasi-soliton and other behaviour of the nonlinear cubic-quintic Schr\"odinger equation, \emph{Canad. J. Phys.} \textbf{64}(1986) 311-315.


\bibitem{D2012}Benjamin, D.: Global well-posedness and scattering for the defocusing $L^2$-critical nonlinear Schr\"{o}dinger equation when $d\geq3$. \emph{J. Amer. Math. Soc.} \textbf{25}, 429-463 (2012)


\bibitem{CV2017}Cirant, M., Verzini, G.: Bifurcation and segregation in quadratic two-populations mean field games systems. \emph{ESAIM Control Optim. Calc. Var.} \textbf{23}(3), 1145-1177 (2017)
%\bibitem{CV2017}M. Cirant, G. Verzini, Bifurcation and segregation in quadratic two-populations mean field games systems. ESAIM Control Optim. Calc. Var. 23, no. 3(2017), 1145-1177.
%\bibitem{CRSW2005}J. Colliander, S. Raynor, C. Sulem, J. D. Wright, Ground state mass concentration in the $L^{2}$-critical nonlinear schrodinger equation below $H^{1}$. Mathematical Research Letters, 12(3)(2005), 357-375.
\bibitem{CMM2011}C$\hat{o}$te, R., Martel, Y., Merle, F.: Construction of multi-soliton solutions for the $L^{2}$-supercritical gkdv and NLS equations. \emph{Rev. Mat. Iberoam.} \textbf{27}(1), 273-302 (2011)

\bibitem{DMM2000}Desyatnikov, A., Maimistov, A., Malomed, B.: Three-dimensional spinning solitons
in dispersive media with the cubic-quintic nonlinearity. \emph{Phys. Rev. E}  \textbf{61}(3), 3107-3113 (2000)

\bibitem{DHR2008}Duyckaerts, T., Holmer, J., Roudenko, S.: Scattering for the non-radial 3D cubic
nonlinear Schr\"{o}dinger equation. \emph{Math. Res. Lett.} \textbf{15}(6), 1233-1250 (2008)

%\bibitem{EGBB1997}Esry, B. D., Greene, C. H., Burke, J. P., Bohn, J. L.: Hartree-fock theory for double condensates. Physical Review Letters 78(19), 3594-3597, 1997.
%\bibitem{FH2021}Fukaya, N., Hayashi, M.: Instability of algebraic standing waves for nonlinear Schr\"odinger equations with double power nonlinearities. Transactions of the American Mathhematical Society, 374, 1421-1447, 2021.
%\bibitem{FO2019}Fukaya, N., Ohta, M. (2019). Strong instability of standing waves with negative energy for double power nonlinear Schr\"odinger equations. Osaka Journal of Mathematics. 56(4): 713-726.
\bibitem{F2003}Fukuizumi, R. (2003). Stability and instability of standing waves for nonlinear Schr\"odinger equations. Tohoku Mathematical Publications. No. 25.
%\bibitem{GFTA2000}Gammal, A., Frederico, T., Tomio, Lauro.,  Abdullaev, F. Kh.: Stability analysis of the D-dimensional nonlinear Schr\"odinger equation with trap and two- and three-body interactions. Phys. Lett. A 267, 305-311, 2000.
\bibitem{GFTLC2000}Gammal, A., Frederico, T., Tomio, L., Chomaz, P.: Atomic Bose-Einstein condensation with three-body interactions and collective excitations. \emph{J. Phys. B}  \textbf{33}, 4053-4067 (2000)


\bibitem{GZ2008}Gan, Z. H., Zhang, J.: Sharp threshold of global existence and instability of standing wave for a Davey-Stewartson system, \emph{Commun. Math. Phys.} \textbf{283}, 93-125  (2008)
\bibitem{GNN1979}Gidas, B., Ni, W. M., Nirenberg, L.: Symmetry and related properties via the maximal principle. \emph{Commun. Math. Phys.} \textbf{68}, 209-243 (1979)
%\bibitem{GNN1981}Gidas, B., Ni, W. M., Nirenberg, L. (1981). Symmetry of positive solutions of nonlinear elliptic equations in $\mathbb{R}^{n}$. Mathematical analysis and applications, Part A. Adv. in Math. Suppl. stud. 7a, Academic press. 369-402.
%\bibitem{GV1979}Ginibre, J., Velo, G. (1979). On a class of nonlinear Schr\"odinger equations. I: The Cauchy problem. J. Funct. Anal. 32: 1-32.
%\bibitem{G1977}Glassey, R. T. (1977). On the blowing-up of solutions to the Cauchy problem for the nonlinear Schr\"odinger equation. J. Math. Phys. 18: 1794-1797.
%\bibitem{GJ2018}Gou, T. Jeanjean, L. (2018). Multiple positive normalized solutions for nonlinear Schr\"odinger systems. Nonlinearity,31, no. 5, 2319-2346.
\bibitem{GSS1987}Grillakis, M., Shatah, J., Strauss, W. A.: Stability theory of solitary waves in the presence of symmetry, I. \emph{J. Funct. Anal.} \textbf{74}(1), 160-197 (1987)
\bibitem{GSS1990}Grillakis, M., Shatah, J., Strauss, W. A.: Stability theory of solitary waves in the presence of symmetry, II. \emph{J. Funct. Anal.} \textbf{94}(2), 308-348 (1990)

%\bibitem{GFTC2000} A. Gammal, T. Frederico, L. Tomio and P. Chomaz, Atomic Bose¨CEinstein condenBsation with three-body intercations and collective excitations, J. Phys. B 33 (2000)
%4053-4067.

%\bibitem{HK2005}Hmidi, T., Keraani, S.: Blowup theory for the critical nonlinear Schr\"{o}dinger equations revisited. \emph{Int. Math. Res. Notices} \textbf{46}, 2815-2828  (2005)
%\bibitem{IT2019}Ikoma, N. Tanaka, K. (2019). A note on deformation argument for $L^{2}$ normalized solutions of nonlinear Schr\"odinger
%equations and systems. Adv. Differential Equations 24, no. 11-12, 609-646.
\bibitem{JJTV2022}Jeanjean, L., Jendrej, J.,   Le, T., Visciglia, N.:  Orbital stability of ground states for a Sobolev critical Schr\"{o}dinger equation, \emph{J. Math. Pures Appl.} \textbf{164}, 158-179 (2022)

\bibitem{JL2021}Jeanjean, L.,  Le, T.: Multiple normalized solutions for a Sobolev critical Schr\"odinger equation. \emph{Math. Ann. }\textbf{384}, 101-134 (2022)

\bibitem{JL2022}Jeanjean, L., Lu, S. S.:  On global minimizers for a mass constrained problem. \emph{Calc. Var. Partial Differ. Equ.} \textbf{61}, 214  (2022)

%\bibitem{JL2022}L. Jeanjean, S. S. Lu. Normalized solutions with positive energies for a coercive problem and application to the cubic-quintic nonlinear Schr\"{o}dinger equation. Mathematical Models and Methods in Applied Sciences, 32(8): 1557-1588, 2022.

%\bibitem{K1959}Kato, T.: Growth properties of solutions of the reduced wave equation with a variable coefficient. \emph{Commun. Pure Appl. Math.} \textbf{12}, 403-425  (1959)

\bibitem{KM2006}Kenig, C. E., Merle, F.: Global well-posedness, scattering and blow-up for the energy-critical, focusing, non-linear Schr\"{o}dinger equation in the radial case. \emph{Invent. Math. } \textbf{166}(3), 645-675  (2006)
%\bibitem{KLT2022} Kfoury P , Coz S L , Tsai T P . Analysis of stability and instability for standing waves of the double power one dimensional nonlinear Schr\"{o}dinger equation[J].Comptes Rendus Math\'{e}matique, 360,  867-892,  2022.

%\bibitem{KA2003}  Y. S. Kivshar and G. P. Agrawal, Optical Solitons: From Fibers to Photonic Crystals, Academic Press, 2003.


\bibitem{K2011}Kawano, S.: Uniqueness of positive solutions to semilinear elliptic equations with double power nonlinearities.\emph{ Differ. Integral Equ.} \textbf{24}(24), 201-207  (2011)

\bibitem{MKV2021}Killip, R., Murphy, J., Visan, M.: Scattering for the cubic-quintic NLS: crossing the virial threshold. \emph{SIAM J. Math. Anal.} \textbf{53}(5), 5803-5812 (2021)

\bibitem{KOPV2017}Killip, R., Oh, T., Pocovnicu, O., Visan, M.: Solitons and scattering for the cubic-quintic nonlinear schr\"{o}dinger equation on $\mathbb{R}^{3}$. \emph{Arch. Rational Mech. Anal.} \textbf{225}(1), 469-548  (2017)



%\bibitem{K1959}T. Kato, Growth properties of solutions of the reduced wave equation with a variable coefficient. Comm, Pure Appl. Math. 12(1959), 403-425.
%\bibitem{KVZ2016}R. Killip, M. Visan , X. Zhang, Finite-dimensional approximation and non-squeezing for the cubic nonlinear Schr\"odinger equation on $R^2$, American Journal of Mathematics, 143(2016), 613-680.
%\bibitem{KTV2009}R. Killip, T. Tao , M. Visan, The cubic nonlinear Schr\"odinger equation in two dimensions with radial data, Journal of the European Mathematical Society 11.6(2009), 1203-1258.




%\bibitem{K1989}Kwong, M. K.: Uniqueness of positive solutions of $\triangle u-u+u^{p}=0$ in
 %   $\mathbb{R}^{n}$. Arch. Rational Mech. Anal. 105(3), 243-266, 1989.

%\bibitem{L2008} Le Coz, S., Standing waves in nonlinear Schr\"{o}dinger equations, Analytical and numerical aspects of partial differential equations. (pp.151-192)Publisher: Walter de Gruyter, BerlinEditors: Etienne Emmrich and Petra Wittbold, 2009.

%\bibitem{CMR2016}Le Coz, S., Martel, Y., Rapha\"el, P. (2016) Minimal mass blow up solutions for a double power nonlinear Schr\"odinger equation. Revista Matem$\acute{a}$tica Iberomamericana. 32(3): 795-833.
\bibitem{LR2020}Lewin, M., Rota Nodari, S.: The double-power nonlinear Schr\"{o}dinger equation and its generalizations: uniqueness, non-degeneracy and applications, \emph{Calc. Var. Partial Differ. Equ.} \textbf{59}(6), 197 (2020)


%\bibitem{LN1993}Li, Y., Ni, W. M. (1993). Radial symmetry of positive solutions of nonlinear elliptic equations in $\mathbb{R}^{n}$. Comm. Partial Differential Equations. 18: 1043-1054.
%\bibitem{LYZ2020}Li, H. W., Yang, Z., Zou, W. M. 2020. Normalized solutions for norlinear Schr\"odinger equations (in chinese)., Sci Sin Math. 50: 1023-1044, doi:10.1360/SSM-2020-0120.
%\bibitem{LL2000}Lieb, E. H., Loss, M.: Analysis, Graduate Studies in Mathematics.  American Mathematical Society, (2000)
\bibitem{L2009}Le Coz, S.:  Standing Waves in Nonlinear Schr\"{o}dinger Equations, Analytical and Numerical Aspects of Partial Differential Equations. Walter de Gruyter, Berlin (2009)
\bibitem{MM2006}Martel, Y., Merle, F.: Multi solitary waves for nonlinear Schr\"{o}dinger equations. \emph{Ann. I. H. Poincar\'{e}-AN.} \textbf{23}, 849-864 (2006)




\bibitem{MMT2006}Martel, Y., Merle, F., Tsai, T. P.: Stability in $H^1$ of the sum of $K$ solitary waves for some nonlinear Schr\"{o}dinger equations. \emph{Duke Math. J.} \textbf{133}(3), 405-466 (2006)
%
\bibitem{M2024}Martel, Y., Asymptotic stability of small standing solitary waves of the one-dimensional cubic-quintic Schr\"{o}dinger equation. \emph{Invent. Math.}, \textbf{237}, 1253-1328 (2024)

%Rapha\"el, P. (2018). Strongly interacting blow up bubbles for the mass critical NLS. Ann. Sci. $\acute{E}$c. Norm. Sup$\acute{E}$r. 51: 701-737.
%\bibitem{M1993}McLeod, K.: Uniqueness of positive radial solutions of $\Delta u+f(u)=0$ in $\mathbb{R}^{n}$. II, Trans. Amer. Math. Soc. 339,  495-505, 1993.
%\bibitem{MCGSQ2002}H. Michinel, J. Campo-T\'{a}boas, R. Garc\'{i}a-Fern\'{a}ndez, J. R. Salgueiro and M. L.Quiroga-Teixeiro, Liquid light condensates, Phys. Rev. E 65 (2002) 066604.

%\bibitem{MMC2002(3)}Mihalache, D., Mazilu, D., Crasovan, L. C., Towers, I. B., Malomed, A., Buryak, A. V., Torner, L., Lederer, F.: Stable three-dimensional spinning optical solitons supported by competing quadratic and cubic nonlinearities. Phys. Rev. E 66, 016613 - 7, 2002.
%\bibitem{MMC2002(2)}Mihalache, D., Mazilu, D., Towers, I., Malomed, B. A., Lederer, F.: Stable two-dimensional spinning solitons in a bimodal cubic-quintic model with four-wave mixing. J. Opt. A-Pure Appl. Opt. 4, 615-623, 2002.

\bibitem{MMCML2000}Mihalache, D., Mazilu, D., Crasovan, L. C., Malomed, B. A., Lederer, F.: Three-dimensional spinning solitons in the cubic-quintic nonlinear medium. \emph{Phys. Rev. E}
\textbf{61}(6), 7142-7145 (2000)
\bibitem{MMCTBMT2002} Mihalache, D., Mazilu, D., Crasovan, L. C., Towers, I., Buryak, A. V., Malomed, B. A., Torner, L.: Stable spinning solitons in three dimensions. \emph{Phys. Rev. Lett.} \textbf{88}(7), 4 (2002)


\bibitem{M2019}Malomed, B.: Vortex solitons: Old results and new perspectives. \emph{Phys. D } \textbf{399}, 108-137  (2019)



\bibitem{MXZ2013}Miao, C.,  Xu, G., Zhao, L.: The dynamics of the 3D radial NLS with the combined terms. \emph{Commun. Math. Phys.} \textbf{318}(3), 767-808  (2013)
%\bibitem{M1990}Merle, F., Construction of solutions with exactly k blow-up points for the Schr\"odinger equation with critical nonlinearity. Comm. Math. Phys. 129 (2), 223-240, 1990.
\bibitem{NS2012}Nakanishi, K., Schlag, W.: Global dynamics above the ground state energy for the cubic NLS equation in 3D. \emph{Calc. Var. Partial Differ. Equ.} \textbf{44}(1-2), 1-45 (2012)
%\bibitem{N2019}Nguy$\tilde{\hat{e}}$n, Ti$\acute{\hat{e}}$n Vinh. (2019). Existence of multi-solitary waves with logarithmic relative distances for the NLS equation. C. R. Acad. Sci. Paris. Ser. I  357: 13-58.
\bibitem{O1995}Ohta, M. (1995). Stability and instability of standing waves for one-dimensional nonlinear Schr\"odinger equations with double power nonlinearity. Kodal Math. J. 18: 68-74.
\bibitem{OT1991}Ogawa, T., Tsutsumi, Y.: Blow-up of $H^{1}$ solutions for the nonlinear Schr\"{o}dinger equation. \emph{J. Differ. Equ.} \textbf{92}: 317-330 (1991)
%\bibitem{OS1998}Ouyang, T., Shi, J. (1998). Exact multiplicity of positive solutions for a class of semilinear problems. J. Differential Equations. 146: 121-156.
%\bibitem{PPVV2019}Pellacci, B., Pistoia, A. Vaira, G., Verzini, G. (2019). Normalized concentrating solutions to nonlinear elliptic
%problems, Journal of Differential Equations.  1-38.
\bibitem{P1965}Pohozaev, S. I.: Eingenfunctions of the equations of the $\Delta u+\lambda f(u)=0$. \emph{Sov. Math. Doklady.} \textbf{165}, 1408-1411 (1965)

\bibitem{PPT1979}Pushkarov, K. I., Pushkarov, D. I. and Tomov, I. V., Self- of light beams in nonlinear media: soliton solutions. \emph{Opt. Quant. Electron.} \textbf{11}(1979) 471-478.


%\bibitem{PPT1979}K. I. Pushkarov, D. I. Pushkarov and I. V. Tomov, Self-action of light beams in
%nonlinear media: soliton solutions, Opt. Quant. Electron. 11 (1979) 471-478.

%\bibitem{RS1978}Reed, M., Simon, B.: Methods of Modern Mathematical Physics. Vol. IV, Academic Press, New York (1978)
%\bibitem{RSS2003}Rodnianski, I., Schlag, W., Soffer, A. (2003). Asymptotic stability of N-soliton states of NLS. Arxiv Cornell University Library . 32.10: 1643-1677.
\bibitem{S2009}Schlag, W.: Stable manifolds for an orbitally unstable NLS. \emph{Ann. of Math.} \textbf{169}(1): 139-227 (2009)
\bibitem{S2021}Soave, N.: Normalized ground states for the NLS equation with combined nonlinearities: The Sobolev critical case. \emph{J. Funct. Anal.} \textbf{279}(6), 108610 (2020).

\bibitem{ST2000}Serrin, J., Tang, M.: Uniqueness of ground states for quasilinear elliptic equations. \emph{Indiana Univ. Math. J.} \textbf{49}(3), 897-923 (2000)


%\bibitem{SS1985}Shatah, J., Strauss, W. A. (1985). Instability of nonlinear bound states. Comm. Math. Phys. 100: 173-190.
%\bibitem{S2020}Soave, N.: Normalized ground states for the NLS equation with combined nonlinearities. J. Differential Equations. 269, 6941-6987, 2020.

\bibitem{S2021}Soave, N.: Normalized ground states for the NLS equation with combined nonlinearities: The Sobolev critical case. \emph{J. Funct. Anal.} \textbf{279}(6), 108610 (2020).
%\bibitem{S1977}Strauss, W. A.: Existence of solitary waves in higher dimensions.\emph{ Commun. Math. Phys.} \textbf{55}(2), 149-162 (1977)
%\bibitem{SS1999}Sulem, C., Sulem, P. L. (1999). The nonlinear Schr\"odinger equation: self-focusing and wave collapse. Springer New York.
%\bibitem{T2006}Tao, T. (2006). Nonlinear dispersive equations. Local and global analysis. CBMS Regional Conference Series in Mathematics. 106.
%\bibitem{TVZ2007}Tao, T., Visan, M., Zhang, X. The nonlinear schrdinger equation with combined power-type nonlinearities. \emph{Commun. Partial Differential Equations}, \textbf{32}: 1281-1343 (2007).
%\bibitem{T1987} Tsutsumi, Y. (1987). $L^{2}$-solutions for nonlinear Schr\"odinger equations and nonlinear group, Funkcial. Ekvac. 30: 115-125.
\bibitem{T2009}Tao, T.: Why are solitons stable? Bulletin of the American Mathematical Society, \textbf{46}(1), 1-33  (2009)
\bibitem{TVZ2007} Tao, T., Visan, M., Zhang, X.: The nonlinear Schr\"{o}dinger equation with combined power-type nonlinearities. \emph{Commun. Partial Differ. Equ.} \textbf{32}(7-9), 1281-1343 (2007)

%\bibitem{TVZ2008}T. Tao, M. Visan,  X. Zhang,  Minimal-mass blowup solutions of the mass-critical NLS, Forum Mathematicum, 20(5)(2008), 881-919.
\bibitem{WW2022}Wei, J. C., Wu, Y. Z.; Normalized solutions for Schr\"{o}dinger equations with critical Sobolev exponent and mixed nonlinearities. \emph{J. Funct. Anal.} \textbf{283}(6), 109574 (2022)
%\bibitem{W1983}Weinstein, M. I.: Nonlinear Schr\"{o}dinger equations and sharp interpolations estimates. \emph{Commun. Math. Phys.} \textbf{87}, 567-576  (1983)
\bibitem{W1985}Weinstein, M. I.: Modulation stability of ground states of nonlinear Schr\"{o}dinger equations. \emph{SIAM J. Math. Anal.} \textbf{16}, 472-491 (1985)
\bibitem{W1986}Weinstein, M. I.: Lyapunov stability of ground states of nonlinear dispersive evolution equations. \emph{Commun. Pure  Appl. Math.} \textbf{39}, 51-68  (1986)
%\bibitem{ZS1972}Zakharov, V. E., Shabat, A. B. (1972). Exact theory of two-dimensional self-focusing and one-dimensional self-modulation of waves in nonlinear media. Soviet Physics JETP. 34: 62-69.
\bibitem{Z2000}Zhang, J.: Stability of attractive Bose-Einstein condensates. \emph{J. Stat. Phys.} \textbf{101}(3-4), 731-746 (2000)
\bibitem{Z2002}Zhang, J.: Cross-constrained variational problem and nonlinear Schr\"{o}dinger equation. Foundations of computational mathematics (Hong Kong, 2000), 457-469, World Sci. Publ., River Edge, NJ, (2002)
\bibitem{Z2005}Zhang, J.: Sharp threshold for blowup and global existence in nonlinear Schr\"{o}dinger equations under a harmonic potential. \emph{Commun. Partial Differ. Equ.} \textbf{30}, 1429-1443 (2005)

\bibitem{Z2006}Zhang, X.: On the Cauchy problem of 3-D energy-critical Schr\"{o}dinger equations with subcritical perturbations. \emph{J. Differ. Equ.} \textbf{230}(2), 422-445 (2006)
\bibitem{ZH1994}Zhou, C., He, X. T.: Stochastic diffusion of electrons in evolutive Langmuir fields. \emph{Phys. Scr.} \textbf{50}, 415-418 (1994)


\end{thebibliography}


\end{document}